\documentclass[11pt, reqno, a4paper]{amsart}

\usepackage{amsmath, amssymb, amsthm}
\usepackage{amscd}
\usepackage[figuresright]{rotating}

\usepackage{color}

\usepackage{etex}
\usepackage{dsfont}
\usepackage[matrix, arrow, curve, frame, arc]{xy}
\usepackage{pgf}
\usepackage[square,sort,comma,numbers]{natbib}
 \usepackage{url}
 \usepackage{hyperref}
 \usepackage[shortlabels]{enumitem} 
\usepackage{multicol}

\newcommand{\tikzAngleOfLine}{\tikz@AngleOfLine}
\def\tikz@AngleOfLine(#1)(#2)#3{%
\pgfmathanglebetweenpoints{%
\pgfpointanchor{#1}{center}}{%
\pgfpointanchor{#2}{center}}
\pgfmathsetmacro{#3}{\pgfmathresult}%
}

\newcommand{\cB}{\mathcal{B}}
\newcommand{\la}{\langle}
\newcommand{\ra}{\rangle}

\newcommand{\La}{\Lambda}

\newcommand{\cA}{\mathcal{A}}
\newcommand{\cS}{\mathcal{S}}

\newcommand{\End}{\operatorname{End}}
\newcommand{\Hom}{\operatorname{Hom}}
\newcommand{\gldim}{\operatorname{gldim}}
\newcommand{\findim}{\operatorname{findim}}
\newcommand{\Ext}{\operatorname{Ext}}
\newcommand{\Tor}{\operatorname{Tor}}
\newcommand{\add}{\operatorname{add}}
\newcommand{\pdim}{\operatorname{pdim}}
\newcommand{\m}{\!\operatorname{-mod}}

\newcommand{\proj}{\!\operatorname{-proj}}
\newcommand{\St}{\Delta}
\newcommand{\Cs}{\nabla}
\renewcommand{\L}{\Lambda}
\renewcommand{\l}{\lambda}

\renewcommand{\top}{\operatorname{top}}

\newcommand{\soc}{\operatorname{soc}}

\newcommand{\injdim}{\operatorname{idim}}
\newcommand{\domdim}{\!\operatorname{-domdim}}

\newcommand{\codomdim}{\!\operatorname{-codomdim}}

\newtheorem{numberingthm}{Theorem}[subsection] 
\theoremstyle{definition}
\newtheorem{Def}[numberingthm]{Definition}
\newtheorem{example}[numberingthm]{Example}

\theoremstyle{plain}
\newtheorem{Prop}[numberingthm]{Proposition}
\newtheorem{Theorem}[numberingthm]{Theorem}

\newtheorem{Cor}[numberingthm]{Corollary}
\newtheorem{Lemma}[numberingthm]{Lemma}
\newtheorem{Remark}[numberingthm]{Remark}
\newtheorem{thmintroduction}{Theorem}

\newenvironment{Example}
{\pushQED{\qed}\example}
{\popQED\endexample}

\theoremstyle{remark}

\usepackage{tikz-cd,circuitikz}

\usetikzlibrary{decorations.pathreplacing}
\usetikzlibrary{decorations.markings, arrows.meta}
\usetikzlibrary{arrows}

\begin{document}

\baselineskip=14pt

\title[Homological dimensions of Schur algebras $S(p,2p)$]{Homological dimensions of Schur algebras $S(p,2p)$ and an Auslander-type correspondence}

\author[T. Cruz]{Tiago Cruz}
\address[Tiago Cruz]{Institut f\"ur Algebra und Zahlentheorie,
Universit\"at Stuttgart, Germany }

\email{tiago.cruz@mathematik.uni-stuttgart.de}

\author[K. Erdmann]{Karin Erdmann}
\address[Karin Erdmann]{Mathematical Institute,
   Oxford University,
       ROQ, Oxford OX2 6GG,
   United Kingdom}
\email{erdmann@maths.ox.ac.uk}

\begin{abstract}
	We study the homological properties of Schur algebras $S(p, 2p)$ over a field $k$ of positive characteristic $p$, focusing on their interplay with the representation theory of quotients of group algebras of symmetric groups via Schur-Weyl duality. Schur-Weyl duality establishes that the centraliser algebra, $\Lambda(p, 2p)$, of the tensor space $(k^p)^{\otimes 2p}$ (as a module over $S(p, 2p)$) is a quotient of the group algebra of the symmetric group. In this paper, we prove that Schur-Weyl duality between $S(p, 2p)$ and $\Lambda(p, 2p)$ is an instance of an Auslander-type correspondence. 
	
	We compute the global dimension of Schur algebras $S(p, 2p)$ and their relative dominant dimension with respect to the tensor space $(k^p)^{\otimes 2p}$. In particular, we show that the pair $(S(p, 2p), (k^p)^{\otimes 2p})$ forms a relative $4(p-1)$-Auslander pair in the sense of Cruz and Psaroudakis, thereby connecting  Schur algebras with higher homological algebra. Moreover, we determine the Hemmer-Nakano dimension associated with the quasi-hereditary cover of $\Lambda(p, 2p)$ that arises from Schur-Weyl duality. As an application, we show that the direct sum of some Young modules over $\Lambda(p, 2p)$ is a full tilting module when $p>2$.
\end{abstract}

\maketitle

\section{Introduction}

The interplay between Schur algebras and symmetric group algebras, due to Schur–Weyl duality, occupies a central role in the representation theory of algebraic groups and related areas such as categorification and invariant theory. In the modular setting, that is, over a field $k$ of positive characteristic $p$, this relationship becomes even richer and more interesting, reflecting deep structural phenomena and properties such as the failure of semi-simplicity, the existence of Hemmer-Nakano type results, quasi-hereditary covers, and non-trivial cellular structures.

In recent years, the role and importance of homological algebra in phenomena like Schur-Weyl duality has been increasing either by the use of computation-free proofs via dominant dimensions (\cite{KSX}) but also via higher versions of this phenomena that appear in the form of quasi-hereditary covers in the sense of Rouquier \cite{Rou} (see for instance \cite{FK} and \cite{Cr2}).

At the intersection of homological algebra with representation theory lies the Auslander correspondence \cite{zbMATH03517355}: a bridge between algebras with nice homological properties called Auslander algebras and algebras of finite representation-type. Our goal in this paper is to connect certain quotients of groups algebras of symmetric groups on $2p$ letters and relate them with Schur algebras (in characteristic $p$) via a generalisation of the Auslander correspondence. Through this connection, we aim to gain insight into the homological structure of Schur algebras of $\operatorname{GL}_p$, the quasi-hereditary covers of these quotients and, on the quotients themselves.


\subsection*{Setup} 

Let $k$ be a field of positive characteristic $p$, and let $V$ be a $k$-vector space of dimension $n$. The symmetric group on $d$ letters $\cS_d$ acts on the right on the tensor space $V^{\otimes d}$, by place permutations, and the centraliser of this action is the Schur algebra $S_k(n, d)$. 
Schur algebras are quasi-hereditary (see for instance \cite{zbMATH04031957} or \cite{zbMATH04116809}), and so they enjoy nice homological properties like possessing finite global dimension (see also \cite{zbMATH04114884}).

If $n\geq d$, then $k\cS_d$ acts faithfully on $V^{\otimes d}$, and the structure of the Schur algebra $S(n, d)$ is relatively well understood (see for instance \cite{zbMATH00966941, FK, zbMATH06551164, HN, Gr, Do}). In contrast, the case $n<d$ is significantly more intricate.
	In general, the action of $\cS_d$ on the tensor space is no longer faithful, and as a result, the Schur-Weyl dual of $S_k(n, d)$ is not self-injective, leading to more subtle homological behavior that remains largely unexplored outside of special cases of the pair $(n, d)$ (see for instance \cite{P, EH, zbMATH02005565, zbMATH06334301, CE24}).  The centraliser of $V^{\otimes d}$ as an $S_k(n, d)$-module is the quotient algebra $\Lambda_k(n, d):=k\cS_d/I_n$, where $I_n$ is the annihilator of the $k\cS_d$-module $V^{\otimes d}$. The algebra $\Lambda_k(n, d)$ carries a cellular structure but little is known about its homological structure. In particular, it is not even known for which parameters $n, d$  the algebra $\Lambda_k(n, d)$ is Iwanaga-Gorenstein (that is, when the minimal injective resolution of the regular module has finite length).

The algebra $\Lambda_k(n, d)$ has a quasi-hereditary cover constructed from the Ringel dual $R_k(n, d)$ of $S_k(n, d)$, as shown in  \cite{Cr2}. Here $R_k(n, d)$ is the endomorphism algebra $\End_{S_k(n,d)}(T)^{op}$, where $T$ is the (basic) characteristic tilting module for $S_k(n, d)$.
The quasi-hereditary cover
comes equipped with a Schur functor
$$F_{n,d}\colon R_k(n, d)\m\rightarrow \L_k(n, d)\m
$$ which, when $p>2$, induces isomorphisms 
	$$\Ext_{R_k(n, d)}^j(M, N)\cong \Ext_{\La_k(n, d)}^j(F_{n, d}M, F_{n, d}N)$$ for every $0\leq j\leq i$ and all modules $M, N$ with a standard filtration for some $i$. The optimal such value $i\in \mathbb{N}\cup \{\infty\}$ is called the Hemmer-Nakano dimension of $\mathcal{F}(\St_{R_k(n, d)})$, where $\mathcal{F}(\St_{R_k(n, d)})$ stands for the subcategory of all modules having a standard filtration in the module category of $R_k(n, d)$. Finding this value in general would unravel the precise homological depth of Schur-Weyl duality.

When $n=2$, the algebras $\L_k(2, d)$ are Temperley-Lieb algebras and we have determined in \cite{CE24}  the Hemmer-Nakano dimension of $\mathcal{F}(\Delta_{R_k(2, d)})$ for
arbitrary $d$. This was done by using tools from the representation theory
of algebraic groups together with tools coming from abstract representation theory and homological algebra. In particular, we extensively used and expanded the relative theory of dominant dimension studied in \cite{Cr2}. In \cite{CP1}, the techniques of relative theory of dominant dimension were used to link the theory of higher Auslander algebras with tilting theory leading to the concept of relative Auslander pairs. Within this language, one of the main results of \cite{CE24} can be interpreted as asserting that $(S_k(2, d), V^{\otimes d})$ forms a relative Auslander pair. In \cite{CP2}, a generalisation of the Auslander correspondence was obtained for relative Auslander pairs. From \cite{CP2} and \cite{CE24}, we obtain then an Auslander-type correspondence between $q$-Schur algebras of the quantum group $G(2)$ and Temperley-Lieb algebras.

For $n>2$ the tools from the representation theory of algebraic groups are not as accessible as in the case $n=2$, and so the approach started in \cite{CE24} cannot be continued for the higher cases $n>2$. Instead, we  exploit that  the representations of the group algebra  $K\cS_{2p}$ are reasonably well understood (see \cite{zbMATH04125660, zbMATH07389875}). With them we illustrate that the situation for $n=p=2$ is not an isolated case and that there are more classes of Schur algebras with homological structures resembling higher Auslander algebras.

\newpage 
\subsection*{Main results}

In this paper, we determine (over any field with characteristic $p$):

\begin{itemize}
	\item  The global dimension of the Schur algebras $S_k(t, 2p)$ for $t=p, \ldots, 2p-1$ using homological properties of Specht modules, Young modules and Schur functors;
	\item The finitistic dimension of the algebras $\Lambda_k(p, 2p)$ for $p>2$.
	\item The Hemmer-Nakano dimension of $\mathcal{F}(\St_{R_k(p, 2p)})$.
	\item The relative dominant dimension of $S_k(p, 2p)$ with respect to $V^{\otimes 2p}$.
	\item An Auslander-type correspondence between the Schur algebras $S(p, 2p)$ and the algebras $\Lambda_k(p, 2p)$.
\end{itemize}

Indeed our main result can be summarised as follows:
\begin{thmintroduction}[see Theorems \ref{thm3:3:2} and \ref{maintheoremA}] \label{TheoremA}
	Let $k$ be an arbitrary field with positive characteristic $p$. Then, the following holds.
	\begin{enumerate}
		\item $\gldim S_k(t, 2p)=4(p-1)$ for every $t=p, \ldots, 2p-1$.
		\item $(S_k(p, 2p), V^{\otimes 2p})$ is a relative $4(p-1)$-Auslander pair, that is, the relative dominant dimension of $S_k(p, 2p)$ with respect to $V^{\otimes 2p}$ is equal to $\gldim S_k(p, 2p)=4(p-1)$.
		\item The algebra $\Lambda_k(p, 2p)$ is Iwanaga-Gorenstein with infinite global dimension and with finitistic dimension equal to $2p-4$ when $p>2$.
	\end{enumerate}
\end{thmintroduction}

In particular, we obtain a new proof to \citep[Theorem 5.9]{zbMATH02005565} in the case $m=2$. 
A consequence of (2) is that the theory developed in \cite{CP2} applies to our situation and so there is an Auslander-type correspondence giving a duality between the pair $(S_k(p, 2p), V^{\otimes 2p})$ and the pair $(\L_k(p, 2p), V^{\otimes 2p})$. Further, it illustrates that the Schur algebras $S_k(p, 2p)$ are the homological side of $\L_k(p, 2p)$ while the algebras $\L_k(p, 2p)$ are the combinatorial side of $S_k(p, 2p)$. 
As an application of (2), we obtain from the relative theory of dominant dimension that the Hemmer-Nakano dimension of $\mathcal{F}(\Delta_{R_k(p, 2p)})$ is exactly $4(p-1)-2=4p-6$. Thus, it follows that $R_k(p, 2p)$ (together with $F_{p, 2p}$) is the best quasi-hereditary cover of $\Lambda_k(p, 2p)$ that sends standard modules to the Specht modules of $\Lambda_k(p, 2p)$. A non-immediate application of (2) is that the direct sum of Young modules labelled by partitions of $2p$ distinct from $(2p, 0, \cdots)$ is a full tilting module over $\L_k(p, 2p)$. Thus, (2) unravels another duality between $S_k(p, 2p)$ and $\L_k(p, 2p)$ through the tensor space. Indeed, $V^{\otimes 2p}$ is a direct summand of a full tilting module as an $S_k(p, 2p)$-module (in the sense of abstract representation theory) while it contains a full tilting module as direct summand as $\L_k(p, 2p)$-module.  Our approach further clarifies and gives further evidence why the Schur algebras $S(p, mp)$ stand out among those of the form $S(n, d)$ with $n<d$, exhibiting better homological properties and appearing to be among the more tractable cases, second only to the well-understood Schur algebras $S(n, n)$.

Our strategy to tackle Theorem \ref{TheoremA} is based on explicit computations of projective resolutions of non-projective Young modules in the principal block of $\Lambda_k(p, 2p)$ as well as coresolutions of certain Specht modules by Young modules. A key ingredient is that $k\cS_{2p-1}$ is of finite representation-type and much of the cellular structure of $k\cS_{2p}$ (that is, the Specht modules and Specht filtrations) can be obtained by inducing from $k\cS_{2p-1}$. In particular, (see for example \cite{zbMATH07389875} which goes back to \cite{zbMATH04125660}), the indecomposable projective modules can be completely described in terms of Specht filtrations. We then apply  techniques of relative theory of dominant dimension, combined with the Schur functor $F_{p, 2p}$ and the simple preserving duality functors in the module categories of $S_k(p, 2p)$ and $\Lambda_k(p, 2p)$ to deduce Theorem \ref{TheoremA}.
As it is common for Schur algebras and Hemmer-Nakano dimensions we treat the cases $p=3$ and $p>3$ separately, however, the underlying methodology and framework remain the same for both cases. The case $p=2$ is not discussed throughout since this was covered mainly in \cite{CE24}.

\subsection*{Outline of the paper} In Subsection \ref{Homological dimensions}, we discuss the homological dimensions that will be used throughout the paper and their associated notation. In Subsection \ref{Auslander pairs}, we give a summary of the properties that relative Auslander pairs have, while in Subsection \ref{Quasihereditaryalgebras}, we provide a short summary for the elementary properties of quasi-hereditary algebras.
In Subsection \ref{sec2dot3}, we give a historical account on the Schur functor $F_{n, d}$ and recall the definitions and notation for the Young and the Specht modules. 
In Subsection \ref{theprincipalblockS2p}, we review the modular representation theory of the principal block of $k\cS_{2p}$ over an algebraically closed field, recalling results presented in \cite{zbMATH07389875}. In particular, we exhibit the Specht filtrations of the indecomposable projective modules and we recall the abacus notation for the partitions to be used throughout. In Section \ref{SpechtandYoungmodulesoverLambda}, we describe the Young and Specht modules of the principal block of $\L_k(p, 2p)$ over algebraically closed fields with positive characteristic distinct from two. In Subsection \ref{reducing the non injective}, we give a bijection between non-injective projective $\L_k(p, 2p)$-modules and non-injective Young $\L_k(p, 2p)$-modules, giving as a byproduct the projective dimension of almost all Young modules. In Theorem \ref{thm3:3:2}, we prove that the principal block of $\Lambda_k(p, 2p)$ is Iwanaga-Gorenstein and we compute its finitistic dimension when $p>3$. In Subsection \ref{coresolvingsingularSpechtmodule}, we construct a coresolution of the singular Specht module by Young modules for the case $p>3$. In Section \ref{Cohomological properties of Young and Specht modules}, we determine the relative dominant dimension of $S_k(p, 2p)$ with respect to the tensor space (see Theorem \ref{maintheoremarbitraryfield}) and deduce as by result that almost all Young modules are self-orthogonal. In particular, in Subsection \ref{tiltingmodulebuiltfromYoungmodules}, we make use of relative dominant dimension to show that the direct sum of all non-semisimple Young modules is a full tilting module in the sense of abstract representation theory (see Theorem \ref{tiltingtheorem}). In Subsection \ref{sec4dot3}, we address the previous results in characteristic three. In Section \ref{homologicalpropertiesofschuralgebras}, we present the main results of the paper. In Theorem \ref{quasihereditarycover}, we compute the Hemmer-Nakano dimension of $\mathcal{F}(\St_{R_k(n, d)})$. In Corollary \ref{simpleYoungmoduledimension}, we show that the simple Young module is the only one with infinite projective dimension. In Subsubsection \ref{quasipreclustertiltingmodule}, we compute the global dimension of the principal block of $S_k(p, 2p)$ over algebraically closed fields, by showing that this block fits into an Auslander pair. In Subsection \ref{globaldimensionofSp2p}, we determine the global dimension of $S_k(p, 2p)$ by showing that the maximal length of projective resolutions occurs in the principal block and we obtain the projective dimension of the tensor space $V^{\otimes 2p}$. By combining the former with Totaro's work \cite{zbMATH00966941}, we conclude that all Schur algebras between $S_k(2p-1, 2p)$ and $S_k(p, 2p)$ have the same global dimension (see Corollary \ref{cor5dot3dot5}). In Subsection \ref{globaldimensionoverarbitraryfields}, we transfer these results on global dimension from Schur algebras over algebraically closed fields to Schur algebras over arbitrary fields proving our main result.

\section{Preliminaries}\label{Preliminaries}
Let $A$ be a finite-dimensional algebra over a field $k$.
 We denote by $A\m$ the category of finitely generated (left) $A$-modules. 
Given $M\in A\m$, we denote by $\add_A M$ (or just $\add M$) the full subcategory
of $A\m$ whose modules are direct summands of a finite direct
sum of copies of $M$. We also denote $\add A$ by $A\proj$. 

The endomorphism algebra of a  module $M\in A\m$ is denoted by ${\rm End}_A(M)$.
We denote by $D$ the \emph{standard duality} functor 
$\Hom_k(-, k): A\m \rightarrow A^{op}\m$ where $A^{op}$ is the opposite algebra of $A$. 
Let $A'$ be a subalgebra of $A$. Given $M\in A\m$, we write  $M\downarrow_{A'}$ to denote the $A'$-module obtained by restriction of scalars. Given $N\in A'\m$ we write $M\uparrow^A$ to denote the induced $A$-module $A\otimes_{A'} M$.
Given $M\in A\m$ we write $\Omega_A^i M$ (or simply $\Omega^i M$) to denote the $i$-th syzygy of $M$ for all $i\in \mathbb{Z}$. We say that a module $M$ is \emph{multiplicity-free} if it decomposes as a direct sum $M=\oplus M_i$   indecomposable modules $M_i$, which are pairwise non-isomorphic.

\subsection{Homological dimensions} \label{Homological dimensions}
We will write $\pdim_A M$, $\injdim_A M$ and $\gldim A$ to denote the projective dimension of $M$, the injective dimension of $M$ and the global dimension of $A$, respectively. We say that an algebra $A$ is \emph{Iwanaga-Gorenstein} if $\injdim_A A$ and $\injdim_{A^{op}} A^{op}$ are finite. In particular, algebras of finite global dimension are Iwanaga-Gorenstein. For Iwanaga-Gorenstein algebras $A$, it is known that $\injdim_A A=\injdim_{A^{op}} A^{op}$ (see for example \cite{zbMATH00067398}).  Moreover, for those, the finitistic dimension $$\findim A:=\sup\{\pdim_A M\colon M\in A\m \text{ with } \pdim_A M<+\infty\}$$ coincides with $\injdim_A A=\injdim_{A^{op}} A^{op}$ (see \citep[VI, Lemma 5.5]{ARS}).
In particular, for algebras of finite global dimension, we have $\injdim_A A=\gldim A$.

  Given $X\in A\m$, we denote by $X^\perp$ the full subcategory $${\{M\in A\m\colon \Ext_A^{i>0}(Z, M)=0, \forall Z\in \add X \}},$$ and by ${}^\perp X$ the full subcategory $\{M\in A\m\colon \Ext_A^{i>0}(M, Z)=0,  \forall Z\in \add X \}$. In addition, given $r\in \mathbb{N}$ we denote by $X^{\perp_r}$ the full subcategory $${\{M\in A\m \colon \Ext_A^{i}(Z, M)=0, \forall Z\in \add X, \  i=1, \ldots, r \}}.$$ Analogously, we write ${}^{\perp_r} X$.

We say that an  $A$-module $Q$ is \emph{self-orthogonal} if $Q\in Q^\perp$. 

 Given two $A$-modules $Q$ and $M$, the \emph{relative dominant dimension} of $M$ with respect to $Q$ is the value $Q\domdim_A M\in \mathbb{N}\cup\{0, +\infty\}$ defined as the supremum of all $n\in \mathbb{N}$ such that there exists an exact sequence \begin{align}
 	0\rightarrow M\rightarrow Q_1\rightarrow Q_2\rightarrow \cdots \rightarrow Q_n \label{equ1}
 \end{align} which remains exact under $\Hom_A(-, Q)$ and $Q_i\in \add Q$ for $i=1, \ldots, n.$ Dually, the \emph{relative codominant dimension} of $M$ with respect to $Q$ is the value $$Q\codomdim_A M:=DQ\domdim_{A^{op}} DM.$$

 Observe that computations of relative dominant dimensions can be reduced to computations involving only multiplicity-free modules. Indeed, if $\add Q=\add Q'$, then it is clear by definition that $Q\domdim_A M=Q'\domdim_A M$ for every $M\in A\m$. Moreover, if $M, N\in A\m$ with $\add M=\add N$, then $Q\domdim_A M=Q\domdim_A N$ (see for instance \citep[Corollary 3.1.9]{Cr2}).
 
 With the following, we can see these computations can be reduced to basic algebras and multiplicity-free modules.
 
  \begin{Lemma}\label{lemma2dotonedotone}
 	Let $F\colon A\m\rightarrow B\m$ be an equivalence of categories and $M, Q\in A\m$. Then $Q\domdim_A M=FQ\domdim_B FM$.
 \end{Lemma}
 \begin{proof}
 	Assume that $Q\domdim_A M\geq n$. Then, there exists an exact sequence of the form (\ref{equ1}), say $\delta$ and $\Hom_A(\delta, Q)$ is exact. It is clear that $F\delta$ is exact and $\Hom_B(F\delta, FQ)\cong \Hom_A(\delta, Q)$ is also exact. Hence, $FQ\domdim_B FM\geq Q\domdim_A M$. The converse inequality is analogous.
 \end{proof}
 
  If $Q$ is a faithful projective-injective $A$-module, then $Q\domdim_A A$ is exactly the \emph{classical dominant dimension} of $A$ (see for instance \citep[7.7]{zbMATH03425791}). Indeed, \cite{zbMATH03425791} is a standard reference for the study of classical dominant dimension.
For a detailed exposition on properties of relative dominant and codominant dimensions, we refer to \cite{Cr2}. See also \citep[Section 3]{CE24}. 

Although most results in \cite{Cr2} are presented using Tor groups, in our setup of finite-dimensional algebras over a field, they can be translated to Ext groups using the following identity.

\begin{Lemma}\label{lemma2dot1dot2}
	 We have $\Ext_A^i(M, N)\cong D\Tor_i^A(DN, M)\cong D\Tor_i^A(DM, N)\cong \Ext_{A^{op}}^i(DN, DM)$ for every $i\geq 0$ and $M, N\in A\m$. 
\end{Lemma}
\begin{proof}
	This is a consequence of Tensor-Hom adjunction and $D$ being an exact contravariant functor.
\end{proof}

Lastly, we need the following relative dimension.
Given an $A$-module $Q$,
we say that a left $A$-module $M$ has \emph{relative $\add Q$-dimension}  $m$ if there exists an exact sequence $$0\rightarrow Q_m\rightarrow Q_{m-1}\rightarrow\cdots \rightarrow Q_0\rightarrow M\rightarrow 0$$ which remains exact under $\Hom_A(Q, -)$ with $Q_i\in \add Q$ and $m$ is the minimal non-negative integer with this property. In such a case, we write $\dim_{\add Q} M=m$. If no such $m$ exists, we write $\dim_{\add Q} M=\infty$.

\subsection{Auslander pairs} \label{Auslander pairs}
The representation theory of Auslander algebras, introduced in \cite{zbMATH03517355}, encodes the representation theory of finite-dimensional algebras of finite representation-type and they have very nice homological properties (see also \citep[VI. 5]{ARS}). These are exactly the algebras satisfying the condition $$ \gldim A \leq 2\leq Q\domdim A$$ for a faithful projective-injective $A$-module $Q$. Iyama in \cite{Iyama07} generalised Auslander algebras to higher Auslander algebras by replacing $2$ by another number $n\in \mathbb{N}\setminus\{1, 2\}$.
Auslander pairs introduced in \cite{CP1} take this concept one step further, by weakening the conditions on $Q$.

\begin{Def}
	Let $Q$ be an $A$-module and $n\in \mathbb{N}$. The pair $(A, Q)$ is called an \emph{$n$-Auslander pair} if it satisfies $\gldim A\leq n\leq Q\domdim_A A$.
\end{Def}

Thus, Auslander algebras and higher Auslander algebras are examples of Auslander pairs (the cases where $Q$ is a faithful projective-injective module).
A classic example of an Auslander algebra is the Auslander algebra of $k[x]/(x^2)$, $\End_{k[x]/(x^2)}(k[x]/(x^2)\oplus k)$. Other toy examples of higher Auslander algebras with global dimension $n$ are for instance the Nakayama algebras with $n+1$ simple modules and with Kupisch series $[2, 2, \ldots, 2, 1].$ Examples of relative $n$-Auslander pairs that are not higher Auslader algebras can be found for instance in \cite{CE24} and \cite{CP1}.

 The following observation implies that we can always assume that $A$ is basic and $Q$ is mutiplicity-free.

\begin{Lemma}
	Morita equivalences preserve $n$-Auslander pairs.
\end{Lemma}
\begin{proof}
	 Indeed, assume that  $F\colon A\m\rightarrow B\m$ is a Morita equivalence. It
	is well known that global dimension is invariant under Morita equivalences, so $\gldim A=\gldim B$.  By Lemma \ref{lemma2dotonedotone}, and since relative dominant dimension does not see multiplicities, we obtain that $Q\domdim_A A=FQ\domdim_B FA=FQ\domdim_B B$. Hence $(A, Q)$ is an $n$-Auslander pair if and only if $(B, FQ)$ is an $n$-Auslander pair.
\end{proof}

If $(A, Q)$ is an $n$-Auslander pair with $n\geq 2$ and $Q$ is a self-orthogonal module with projective dimension $t$ and injective dimension $r$, then $Q$ as module over its endomorphism algebra is a $t$-quasi-generator and $r$-quasi-cogenerator (see \cite{MTcorrespondence} and \cite{CP2}). 

A module $M$ is called an \emph{$t$-quasi-generator} as an $A$-module if $t$ is the minimal non-negative integer such that there exists an exact sequence $0\rightarrow A\rightarrow M_0\rightarrow M_1\rightarrow \cdots M_t\rightarrow 0$ which remains exact under $\Hom_A(-, M)$ and $M_i\in \add M$. When $t=0$, we recover the concept of generator. A module is an \emph{$r$-quasi-cogenerator} as an $A$-module if its dual is an $r$-quasi-generator as an $A^{op}$-module.

In \cite{CP2} a  correspondence was obtained between Auslander pairs $(A, Q)$ with $$\gldim A\geq \pdim_A Q+\injdim_A Q+2$$ and pairs $(B, M)$ where $M$ is a \emph{$(n, t, r)$-quasi-precluster tilting module} as a $B^{op}$-module satisfying $\add M={}^{\perp_r} M\cap M^{\perp_{n-r-2}}\subset B^{op}\m$ with $(n, t, r):=(\gldim A, \pdim_A Q, \injdim_A Q)$. These objects are then a generalisation of precluster tilting modules in the sense of \cite{zbMATH06833443} which were already a generalisation of higher cluster tilting modules in the sense of \cite{Iyama07}. Such a correspondence is a generalisation of Auslander's correspondence and Iyama's correspondence. We refer to \cite{CP1, CP2} for more details.


\subsection{Quasi-hereditary algebras} \label{Quasihereditaryalgebras} Let $A$ be a finite-dimensional algebra over a field $k$ and assume that $(\L, \leq)$ is a poset labelling the isomorphism classes of projective indecomposable $A$-modules $P(\l)$. Denote by $\top M$ the top of the module $M\in A\m$. Denote by $I(\l)$ the injective hull of $\top P(\l)$, $\l\in \L$.
For each $\l\in \L$, define $\St(\l)$ to be the maximal quotient module of $P(\l)$ having only composition factors $\top P(\mu)$ with $\mu\leq \l$. With these, $(A, \{\St(\l)\colon \l\in \L\})$ is called a \emph{split quasi-hereditary $k$-algebra} if the regular module $A$ has a finite filtration with quotients isomorphic to $\St(\l)$ with $\l\in \L$ and  $\End_A(\St(\mu))\cong k$ for every $\mu\in \L$. 

The modules $\St(\l)$ are called \emph{standard} modules.
Given a set of modules $\theta$, we denote by $\mathcal{F}(\theta)$ the subcategory of $A\m$ whose modules admit a finite filtration with quotients isomorphic to modules in $\theta$. 

Dually, associated to each split quasi-hereditary $R$-algebra there exists a set of \emph{costandard} modules $\{\Cs(\l)\colon \l\in \L\}$ in $A\m$ satisfying $$\mathcal{F}(\Cs)=\{N\in A\m\colon \Ext_A^1(M, N)=0, \forall M\in \mathcal{F}(\St)\}.$$ In particular, $\mathcal{F}(\Cs)$ is a coresolving subcategory of $A\m$, that is, closed under cokernels of monomorphisms, closed under extensions, closed under direct summands, and it contains all injective modules. Further, the quasi-hereditary structure imposes that $\Hom_A(\Cs(\alpha), \Cs(\beta))\neq 0$ is a sufficient condition for $\alpha\geq \beta$ with $\alpha, \beta\in \L$.
 An important property of split quasi-hereditary algebras is the existence of the characteristic tilting module. 
 Here, we say that a module $T$ is a \emph{characteristic tilting module} if $\add T=\mathcal{F}(\St)\cap \mathcal{F}(\Cs)$. In particular, $T$ has $|\L|$ distinct indecomposable direct summands, $T(\l)$, $\l\in \L$, and $[T(\l):\top P(\l)]=1$.
Let $T$ be a basic characteristic tilting module. The \emph{Ringel dual} of a quasi-hereditary algebra is, up to Morita equivalence, the basic algebra $\End_A(T)^{op}$. It has also a quasi-hereditary structure with standard modules $\Hom_A(T, \Cs(\l))$, $\l\in \L^{op}$. More details can be found on \cite{DR, DKV} and the references therein.

\subsubsection{Simple preserving dualities} 

Cellular algebras were introduced by Graham and Lehrer in \cite{zbMATH00871761} to study associative algebras endowed with a cell datum, that is, a distinguished basis compatible with an involution and a partial order. Such an involution induces a simple preserving duality on the module category
of the cellular algebra.
For further details on cellular algebras we refer for example to \cite{zbMATH01218863, zbMATH00871761, zbMATH01384521, zbMATH02152778} and the references therein.
 Many quasi-hereditary algebras also admit a simple preserving duality. In \cite[Corollary 4.2]{zbMATH01218863}, a sufficient condition is established under which a quasi-hereditary algebra admitting a simple preserving duality is cellular. Conversely, in \cite{zbMATH01384521}, Koenig and Xi show that a cellular structure on $A$ induces a quasi-hereditary structure on $A$ precisely when the global dimension of $A$ is finite. Moreover, Coulembier in \cite{zbMATH07203140} showed that in such a case the quasi-hereditary structure on $A$ is essentially unique: it is exactly the one generated by the cellular structure.

\subsection{Recap on the Schur functor, Schur algebras and symmetric groups}\label{sec2dot3}

Given a natural number $d$, we denote by $\cS_d$ the symmetric group on $d$ letters. Let $k$ be a field and $V$ an $n$-dimensional vector space. The symmetric group $\cS_d$ acts on the right of the tensor space $V^{\otimes d}$ by place permutation. The \emph{Schur algebra} $S_k(n, d)$ is defined as the endomorphism algebra $\End_{k\cS_d}(V^{\otimes d})$. Schur algebras  admit a quasi-hereditary structure, where the underlying poset is  the set of partitions of $d$ in at most $n$ parts, $\L^+(n, d)$, with the dominance order. Schur algebras and their Ringel duals also admit a cellular structure and a simple preserving duality (see for example \cite[Section 5]{cellularqhalgebras} and the references therein). Hence, the quasi-hereditary structure of Schur algebras is essentially unique (see also \citep[Theorem 2.1.1]{zbMATH07203140}). The standard modules $\St(\l)$ are known as \emph{Weyl modules}  and the costandard modules $\Cs(\l)$ are known as \emph{dual Weyl modules}. Let $T$ be the (multiplicity-free) characteristic tilting module of $S_k(n, d)$.

Denote by $R_k(n, d)$ the Ringel dual of $S_k(n, d)$ and by $\St_{R_k(n, d)}(\l), \Cs_{R_k(n, d)}(\lambda), T_{R_k(n, d)}(\l), \l\in \L^+(n, d)^{op}$, their standard, costandard and indecomposable partial tilting modules, respectively.
We can consider the Schur functor
$$F_{n, d}=\Hom_{R_k(n, d)}(\Hom_{S_k(n, d)}(T, V^{\otimes d}), -)\colon R_k(n, d)\m\rightarrow \End_{S_k(n, d)}(V^{\otimes d})^{op}\m.$$
By \citep[4.3]{E1}, the endomorphism algebra $\End_{S_k(n, d)}(V^{\otimes d})^{op}$ is actually isomorphic to $k\cS_{d}/I_n$ where $I_n$ is the annihilator of $V^{\otimes d}$.

In \cite{Cr2}, the following result was obtained.

\begin{Theorem}\label{thm2dot2dot1}Let $k$ be a field with positive characteristic $p$ and assume that $d\geq p$. Write $\La_k(n, d)=\End_{S_k(n, d)}(V^{\otimes d})^{op}$. Then, the following assertions hold.
	\begin{enumerate}[(a)]
		\item The restriction of the functor $F_{n, d}$ to $R_k(n, d)\proj$ is fully faithful.
		\item There are isomorphisms $\Ext_{R_k(n, d)}^i(M, N)\cong \Ext_{\La_k(n, d)}^i(F_{n, d}M, F_{n, d}N)$ for every \linebreak\mbox{$0\leq i \leq p-3$} and $M, N\in \mathcal{F}(\St_{R_k(n, d)})$.
		\item The functor $L=\Hom_{S_k(n, d)}(V^{\otimes d}, -)\colon S_k(n, d)\m\rightarrow \La_k(n, d)\m$ induces isomorphisms
		$$\Ext_{S_k(n, d)}^i(M, N)\cong \Ext_{\La_k(n, d)}^i(LM, LN)$$ for every $0\leq i\leq p-3$ and $M, N\in \mathcal{F}(\Cs)$.
	\end{enumerate}
\end{Theorem}
\begin{proof}
	Assertions (a) and (b) follow from \citep[Theorems 8.1.3, 8.1.2]{Cr2} while assertion (c) follows from \citep[Theorems 8.1.2 and 5.3.1(b)]{Cr2}.
\end{proof}

When $k$ is a field with characteristic zero or $p>d$, then the above functors are equivalence of categories and the algebras involved are semi-simple.  We also observe that the bounds in the above theorem are not necessarily optimal (see for example \cite{CE24}) and finding their optimal value remains an open problem in most cases when $n<d$. 

Recall from \cite{FK} that the \emph{Hemmer-Nakano dimension} of $\mathcal{F}(\St_{R_k(n, d)})$ (with respect to the Schur functor $F_{n, d}$) is the maximal value $n\in \mathbb{N}\cup \{\infty\}$ so that there are isomorphisms $\Ext_{R_k(n, d)}^i(M, N)\cong \Ext_{\La_k(n, d)}^i(F_{n, d}M, F_{n, d}N)$ for every $0\leq i \leq n$ and $M, N\in \mathcal{F}(\St_{R_k(n, d)})$. 

However, for $n\geq d$ the setup is well known, the optimal value is known and it goes back to the work of Hemmer and Nakano as we recall in the following remark. 

\begin{Remark}
	If $n\geq d$, the Schur algebra $S_k(n, d)$ is Ringel self-dual (see \cite{Do}), $V^{\otimes d}$ is a faithful projective-injective module and $\add_{R_k(n, d)} \Hom_{S(n, d)}(T, V^{\otimes d})=\add_{S_k(n, d)} V^{\otimes d}$.  So, when $n\geq d$, $F_{n, d}$ is equivalent to the classical Schur functor $$F=\Hom_{S_k(d, d)}(V^{\otimes d}, -)\colon S_k(d, d)\m\rightarrow k\cS_d\m.$$ In such a case, Theorem \ref{thm2dot2dot1} is one of the main results of \cite{HN} (see also \cite{FK}).
\end{Remark}

We consider the following modules over $k\cS_{d}$:
 \begin{itemize}
 	\item The \emph{Young permutation module} $M^\lambda=F S^\lambda V$, $\lambda=(\lambda_1, \ldots, \lambda_d)\in \L^+(d, d)$ where $S^\lambda V$ denotes the ($\lambda$-th generalised) symmetric power $S^{\lambda_1}V\otimes \cdots \otimes S^{\lambda_d}$;
 	\item The \emph{Young module} $Y^\l=FI(\l)$, $\l\in \L^+(d, d)$; 
 	\item Any direct sum of modules of the form $Y^\l$ is called a \emph{Young module};
 	 \item The \emph{Specht module} $S^\l=F\Cs(\l)$, $\l\in \L^+(d, d)$ (see \citep[3.8(4)]{E1});
 \end{itemize}

Hence, if $k$ has characteristic zero or $p\geq 3$, both $Y^\lambda$ and $S^\lambda$ are indecomposable modules.
This approach to define Specht, Young and Young permutation modules is inspired by the point of view extensively used in \cite{Do2}  and in \cite{E1}. The Young permutation module $M^\l$  is also isomorphic to $k\cS_{d}\otimes_{k\cS_{\l}} k$, where $k$ is the trivial module over the group algebra $k\cS_d$ of the Young subgroup $\cS_\lambda$ (see for example \citep[2.1 (20)]{Do2}).
An alternative way to regard $M^\l$ is to view it as the permutation module on the set of all left cosets representatives of the Young subgroup $S_\l$ (see for example \citep[Corollary 3.4]{zbMATH01361790}).

Equivalent ways to define Specht and Young modules can be found, for example, in \citep[1.6]{zbMATH00971625}, \cite{MR513828}, \citep[Chapter 7]{MR644144}, \citep[4.1]{zbMATH02153663}.

Actually using the theorem above the second author showed in \citep[Propositions 5.1 and 5.2]{E1} the following identifications \begin{align}
	Y^\l\cong \Hom_{S_k(n, d)}(V^{\otimes d}, I(\l))\cong F_{n, d}T_{R_k(n, d)}(\lambda) \label{eqone} \\ S^\l\cong \Hom_{S_k(n, d)}(V^{\otimes d}, \Cs(\l))\cong F_{n, d} \St_{R_k(n, d)}(\l)  \label{eqtwo}
\end{align} for every $n, d$ and $\l\in \L^+(n, d)$. Here, the underlying phenomenon that helps to understand the origin of these identifications is Ringel duality. Indeed, the functors $\Hom_{S(n, d)}(V^{\otimes d}, -)$ and {$\Hom_{R_k(n, d)}(\Hom_{S(n, d)}(T, V^{\otimes d}), - )\circ \Hom_{S(n, d)}(T, -)$} are isomorphic when restricted to $\mathcal{F}(\Cs)$. An important property that Young modules possess is that $Y^\lambda$ is the unique direct summand of $M^\lambda$ that contains $S^\lambda$. If $k$ has characteristic $p>0$, the Young module $Y^\lambda$ is projective precisely when $\lambda$ is the conjugate of a $p$-regular partition of $d$ (see for instance \citep[Page 103]{zbMATH00971625} or \citep[page 653]{zbMATH01400114}).

We note however that some sources present Young and Specht modules as right modules while other sources use left modules. This causes no conflict because right $k\cS_d$-modules can be viewed as right $k\cS_d^{op}$-modules using that $k\cS_d^{op}$ is isomorphic to $k\cS_d$ via the anti-isomorphism $\sigma\mapsto \sigma^{-1}$. Given a right $k\cS_d$-module $X$, we write $X^{\iota}$ to denote the left $k\cS_d$-module obtained by twisting with $\sigma\mapsto \sigma^{-1}$. Similarly, we write ${}^{\iota} X$ for left modules $X$. Given $X\in k\cS_d\m$, we write ${}^\natural X=D ({}^{\iota} X)$. This way, the functor ${}^\natural(-)\colon k\cS_d\m\rightarrow k\cS_d\m$ is a simple preserving duality and it is the one inherited from the cellular structure of $S_k(d, d)$.
Similarly, we write $X^\natural$ for right modules.

Thus, under the above duality, the Young modules are self-dual (that is, ${}^\natural (Y^\lambda)\cong Y^\lambda$) since the indecomposable direct summands of the characteristic tilting module of $S_k(d, d)$ are. In particular, ${}^\iota (Y^\lambda)$ is the right Young module labelled by $\lambda$. Hence, the properties of right Young modules are exactly the same as the ones for left Young modules.
The modules ${}^\natural(S^\lambda)$ are known as \emph{dual Specht} modules. Unless stated otherwise, we use the left Young and Specht modules.

For this reason, we will also consider $DV^{\otimes d}$ as a left $k\cS_d$-module since $V^{\otimes d}$ comes naturally with a right $\cS_d$-action. It follows by their definition that the indecomposable direct summands of $DV^{\otimes d}\cong \Hom_{S_k(n, d)}(V^{\otimes d}, DS_k(n, d))$ are precisely the Young modules $Y^\lambda$ with $\lambda\in \Lambda^+(n, d)$. As left $S_k(n, d)$-modules, 
the indecomposable direct summands of $V^{\otimes d}$ are precisely the modules $T(\lambda)$, where in characteristic zero all $\lambda\in \Lambda^+(n, d)$ occur, while in characteristic $p>0$ only those $\lambda$ that are $p$-regular occur.

We will use the results only for classical Schur algebras, however we note
that the results above also hold over commutative Noetherian rings and also for $q$-Schur algebras, up to some technicalities. For simplicity, we present above only the versions for the classical case of Schur algebras.

We also point out that defining Young modules and Specht modules using the functors $F_{n, d}$ and $F$ is advantageous to understand the meaning behind the labellings that we use. Indeed, this way it is easier to understand and connect the many distinct ways to label, for instance, the simple modules over the symmetric groups. In this terminology, the question posed in \cite{zbMATH03558048} is explained by Ringel self-duality of the Schur algebra $S_k(d, d)$. In fact, one labelling is obtained by applying $F$ to the simple modules labelled by the quasi-hereditary structure of $S_k(d, d)$ while the other by applying $F_{d, d}$ to the simple modules by the quasi-hereditary structure of the Ringel dual (in this specific case: the opposite algebra).


\subsection{The principal block of 
	$\cS_{2p}$ } \label{theprincipalblockS2p}


The following is based on \citep[Appendix B]{zbMATH07389875}.  We will give 
a brief outline and refer to this appendix for further details.

Let $B$ be the basic algebra of the principal block of $k\cS_{2p}$ where $p\geq 3$ and $k$ has characteristic $p$. For our purposes, we do not need to discuss the case $p=2$. We assume that $k$ is algebraically closed unless stated otherwise. In the usual terminology, $B$ has
$p$-core $\emptyset$ and (combinatorial) weight $w=2$.
 Young modules, Specht modules, and the structure of indecomposable projective modules for blocks of symmetric groups of weight two have been extensively studied; see for example \cite{zbMATH04125660, zbMATH00509866, zbMATH00023610, zbMATH00780812, zbMATH01400114, zbMATH01819312}.

 The labelling set for the Specht modules is the set of all the partitions of $2p$, and the labelling set for the simple modules is the set of all the $p$-regular partitions of $2p$. 
A Specht module  belongs to this block (that is, '$\lambda$ belongs to the block') \  if and only if one may remove two (rim) $p$-hooks from the Young diagram $[\lambda]$.

\subsubsection{Abacus notation}\label{abacusnotation}

We use the abacus notation for partitions in this block. This  is both convenient but also provides extra information. 
Namely, it leads to an easy description of the Gabriel quiver of the block, and also shows how to identify  modules with Specht filtration which are
induced from Brauer tree algebras of $k\cS_{2p-1}$.

		  We take an abacus $\Gamma$ with $p$ runners, with two beads on each runner. We may display all partitions $\lambda$ in $B$ on this fixed abacus.
		 
		 We label positions on $\Gamma$ from left to right, then top to bottom, starting with $0$. The runners are labelled by the numbers $1, 2, \ldots, p$.
		 
		 Let $\lambda = (\lambda_1, \ldots,  \lambda_s)$ be a partition in $B$, we note that $2p$ is an integer  greater than or equal to $s$ for all such $\lambda$.
		 The abacus display $\Gamma_{\lambda}$ is given by placing a bead
		 precisely in positions $\beta_i$ where 
		 $$\beta_i = \left\{\begin{array}{ll} \lambda_i-i+ 2p & 1\leq i\leq s\cr -i+2p & i>s
		 \end{array}
		 \right. .
		 $$
		 In other words, the bead is placed at $\beta_i$ so that the number of gaps before the bead is equal to $\lambda_i$. 
		 
		 Moving a bead up on its runner one place corresponds to removing a rim $p$-hook from $[\lambda]$. So if we move each bead up on its
		 runner as far as possible, we get the abacus display for the $p$-core of $\lambda$. 
		 This makes it easy to identify the $p$-core of a partition. 
		 The block $B$ has $p$-weight $2$, that is, in total there are exactly two beads that can be moved up on their respective runner in $\Gamma_{\lambda}$.
		 We have three possibilities:
		 \begin{itemize}
		 	\item[(a)] There is a bead on some runner $v$ that can be moved two places up. We denote the
		 	$\lambda$ by $\la v\ra$.
		 	\item[(b)] There are $1\leq v  < u \leq p$ such that there is a movable bead on runner $v$ and a movable bead on runner $u$. In such a case, we denote $\lambda$ by $\la u, v\ra$. 
		 	\item[(c)] There is some runner $v$ that has a gap followed by two consecutive beads. This one can first move the upper bead one position up, then
		 	the lower bead. In this case, we denote $\lambda$ by $\la v, v\ra$.
		 \end{itemize}

			\begin{Example}\label{examplegabrielquivers2p}
			Using the abacus notation, the relevant poset for $k\cS_{2p}$ in characteristic $p=5$ is the following
			
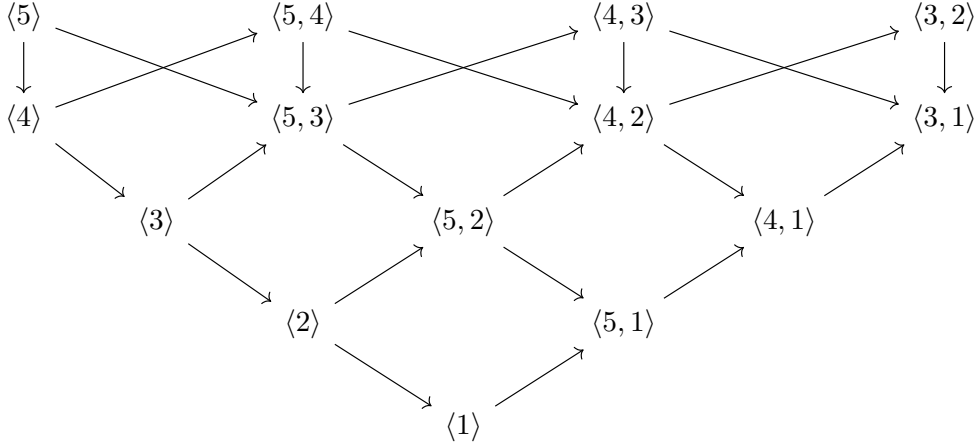
\begin{equation*}
				\begin{tikzcd}
					\langle 5\rangle \arrow[d] \arrow[drr] & &\langle 5, 4\rangle \arrow[d] \arrow[drr]  & & \langle 4, 3\rangle \arrow[d] \arrow[drr]& &  \langle 3, 2\rangle \arrow[d]  \\
					\langle 4\rangle \arrow[urr]  \arrow[dr]&  &\langle 5, 3\rangle \arrow[dr] \arrow[urr] & & \langle 4, 2\rangle \arrow[dr] \arrow[urr]& & \langle 3, 1\rangle \\
					& \langle 3\rangle \arrow[dr] \arrow[ur]  & & \langle 5, 2\rangle \arrow[dr] \arrow[ur]  & &\langle 4, 1\rangle \arrow[ur] \\
					& & \langle 2\rangle \arrow[ur] \arrow[dr]& & \langle 5, 1\rangle \arrow[ur] & \\
					& & & \langle 1\rangle \arrow[ur] &   
				\end{tikzcd}
			\end{equation*}
		The arrows illustrate the dominance order. Indeed, here an arrow $\lambda \rightarrow \beta$ means that $\lambda\geq \beta$ in the dominance order.
		If we replace the directed arrows by two-sided arrows, then we obtain the Gabriel quiver of the  basic algebra of the principal block of $k\cS_{10}$.
		\end{Example}
		
		\begin{Lemma}
			The following inequalities hold under the dominance order:
			\begin{enumerate}
				\item $\langle u\rangle \geq \langle v \rangle$ if and only if $u\geq v$;
				\item $\langle u, v \rangle > \langle u, v-1\rangle$;
				\item $\langle u, v \rangle > \langle u-1, v \rangle$ when $u-1>v$;
				\item $\langle u\rangle > \langle p, u\rangle$ for all $u\neq p$. 
			\end{enumerate}
		\end{Lemma}
	\begin{proof}
		For $u=1, \ldots, p$, the partition $\langle u \rangle$ is of the form $t_u+1^{p-u}$, where $t_u=2p-p+u=p+u$. So, (1) follows.

		The operation of interchanging beads from the runner $v$ in the third row to the runner $v-1$ in the third row does not change the number of gaps before the last bead. But, in $\langle u, v\rangle$ the number of gaps before the second to last bead is bigger than the number of gaps before the second to last bead in $\langle u, v-1\rangle.$ So, (2) follows.
		
			Assume that $u-1>v$.	The number of gaps before the last bead in $\langle u-1, v \rangle$ is equal to the number of gaps before the last bead in $\langle u, v \rangle $ minus one. So, (3) follows.
			
			Suppose now that $u<p$. Then, the number of gaps before the last bead in $\langle p, u\rangle$ is $p$. Since the last bead in $\langle u\rangle$ appears in the fourth row, the number of gaps before the last bead is bigger than $p$. Hence $\langle u\rangle > \langle p, u\rangle$.
	\end{proof}

In the Ext-quiver of $k\cS_{2p}$ over characteristic $p$, only $p$-regular partitions can label the vertices. Hence, Example \ref{examplegabrielquivers2p} illustrates that the partitions $\langle 2, 1\rangle$ and $\langle u, u\rangle$ are not $p$-regular.

\begin{Lemma}\label{singularabacuspartitions}
	A partition $\lambda$ of $2p$ with at most $p$ parts is $p$-singular if and only if it coincides with $\langle 2, 1\rangle $ or it is equal to $\langle u, u\rangle $ for some $u\in\{1, \ldots, p\}$.
\end{Lemma}
	\begin{proof}
		The partition $\langle 2, 1 \rangle$ corresponds to the partition $2^p$ which is $p$-singular. The other $p$-singular partitions of $2p$ contain the part $1$ more than $p-1$ times. In abacus notation, this means that there are $t$ consecutive beads after exactly one empty gap, where $t\geq p$. Since there are only $2p$ beads the gap either occurs in the first or in the second row. But it cannot occur in the second row, since the only chance would be in the first runner, but such an abacus would yield a partition of $p$. So, the gap must occur in the first row, say in the runner $s$. Since there are exactly 2 beads in each runner and only two vertical movements are allowed, the two beads in the runner $s$ are placed in the second and third row. Thus, the partition coincides with $\langle s, s\rangle$.
	\end{proof}
	

		\subsubsection{Specht modules of $k\cS_{2p}$} 
		
		Consider a Specht module $S^{\la u, v\ra}$ labelled by the partition
		$\la u, v\ra$. When $\la u, v\ra$ is $p$-regular, it has a unique simple quotient denoted
		by $D^{\la u, v\ra}$,  and this gives
		a full set of simple $B$-modules (see for example \citep[Theorem 7.1.14 and Corollary 7.1.11]{MR644144}). We write $P^{\la u, v\ra}$ to denote the projective cover of $D^{\la u, v\ra}$.
		The $p$-regular partitions correspond, therefore, to the vertices of
		the Gabriel quiver of the (basic algebra of the) block. We remind the reader that the labelling/notation is different from $D_{\la u, v\ra}$ used for instance in \cite{zbMATH02153663}, that is, the simple module $D^{\la u, v\ra}$ is not necessarily the same as the one $D_{\la u, v\ra}$ used in \cite{zbMATH02153663}.
	
The Loewy structure for the Specht modules is completely described in \citep[Corollary B.6]{zbMATH07389875}.  In the most general case 
when $u\neq v$ and $u-v\geq 2$, the  
composition factors are labelled by  the vertices of the 'mesh' with right corner equal to
$\la u, v\ra$ and left corner $\la u+1, v+1\ra$ (or $\la v+1\ra$).
Here $D^{\la u, v\ra}$ and $D^{\la u+1, v+1\ra}$ are the top and the socle 
of the Specht module, and ${\rm rad} S^{\la u, v\ra}/S^{\la u+1, v+1\ra}$ is
semisimple. We highlight the following two situations.

	\begin{Remark} \label{rmk2dot4dot4} \  \begin{enumerate}[(a)]
			\item If $2\leq s\leq p-2$ then $S^{\la p, s\ra}$ has a submodule isomorphic to $S^{\la s\ra}$: 
			We know the submodule structure of $S^{\la p, s\ra}$ from the discussion above and we see that it  has an indecomposable submodule of
			length two with top isomorphic to $D^{\la s\ra}$ and socle $D^{\la s+1\ra}$. The Specht module $S^{\la s\ra}$ is such a module, and from the Gabriel quiver we know that ${\rm Ext}_{k\cS_{2p}}^1(D^{\la s\ra }, D^{\la s+1 \ra})$ is 1-dimensional. Hence, there is a unique such module up to isomorphism, that is, $S^{\la p, s\ra}$ has a submodule isomorphic to $S^{\la s\ra}$. 
			\item The module $S^{\la p, p-1\ra}$ has length two, with socle  isomorphic to $D^{\la p-1\ra}$ (see \citep[Corollary B6]{zbMATH07389875}).
		\end{enumerate} 
	\end{Remark}

	\subsubsection{Restricting and inducing between  blocks of $k\cS_{2p-1}$ and $B$}
	
	Restricting a Specht module $S^{\lambda}$ to $k\cS_{2p-1}$ gives a module with
	Specht filtration, where the labels for the Specht subquotients are 
	described by the branching rule (see \citep[page 59]{MR644144}).  Removing a node from $[\lambda]$ means
	that on the abacus we  move  one bead to the left (assuming there
	is a gap). 
	The blocks of $k\cS_{2p-1}$ which contain Specht modules which occur
	in this way are all of finite representation-type, that is, they are
	Brauer tree algebras (or are simple).
	
	Here we are interested in the non-simple blocks. They have $p$-cores that are the hook partitions
	$(p-1), \ (p-2, 1), \ldots, (1^{p-1})$; the block with $p$-core $(s-1, 1^{p-s})$ is denoted by $B_s$ in \cite{zbMATH07389875}.
	
	If the restriction of  a Specht module in $B$ to a block $B_s$ is 
	non-zero, then it is a Specht module, and if we induce this to $B$, we
	obtain a module which is indecomposable (which is explained below). It has precisely two Specht subquotients, and the labels can be identified via the abacus:
	they are all partitions obtained by moving a bead one place to the right.
	
	\begin{Def} A  module $U$ in $B$ is called  \emph{Brauer induced} if there is a Specht module $S$ in some Brauer tree algebra  of $k\cS_{2p-1}$ and
		$U\cong S\uparrow^B$.
	\end{Def}

	Then such a module 
	$U$   has precisely two Specht subquotients, and it must be indecomposable (which we explain below).
	We denote it  by $U= U{\lambda \choose \mu}$ if
	it has a submodule isomorphic to $S^{\mu}$ with quotient $S^{\lambda}$.

	We consider one of the blocks $B_s$. If we label the Specht modules in it  by  
	$S_i= S^{\lambda_i}$ for $0\leq i\leq p-1$ with $\lambda_i > \lambda_{i+1}$ 
	in the dominance order, then $S_0$ is a Young module. It follows that  $S_0\uparrow^B$ is also  a Young module.

Let $0\leq i \leq p-2$.  The projective covers, $P_i$, of the Specht modules in $B_s$ satisfy
$$0 \to S_{i+1} \to P_i \to S_i \to  0.
$$
Inducing such a sequence to $B$ gives the projective cover for the Brauer induced module $U = S_i\uparrow^B$. 
One can show that the modules $P_i\uparrow^B$ are indecomposable. This  implies that a Brauer induced module
$S_i\uparrow^B$ has a simple  top, since it is a factor module of an indecomposable projective.
It also has a simple socle; if $i\leq 1$, then it is a submodule of the projective and injective module $P_{i-1}\uparrow^B$ and for $i=0$ it is the Young module which has
a simple top and is selfdual and hence has a simple socle. In particular, it is indecomposable.

	 \subsubsection{Restricting to Brauer tree algebras of $k\cS_{2p-1}$}
	 \label{restrictingtoBrauertreealgebras}
	 
	 We assume now for the remainder of the subsection that $p\geq 5$. We will deal later with the case $p=3$ separately, due to notational reasons.
	 
	 \begin{enumerate}[(1)]
	 	\item \textit{We will first restrict to the block  with core $(p-1)$, that is, $B_p$.}
	 	
	 		 We consider the subquiver of the Gabriel quiver of the basic algebra of $k\cS_{2p}$ with vertices $\la p \ra, \ \la p, p-2\ra, \dots \la p, 1\ra$ and we consider one of these partitions on the abacus. 
	 	We can move a bead from runner $p$ to runner $p-1$. This gives a partition of $2p-1$  which has weight $w=1$ and has $p$-core $(p-1)$.
	 	In this way, we get $p-1$ distinct partitions which form a total order under the dominance order. In addition,  we have the $p$-singular partition $\la p, p\ra$ and if we move the last bead on runner $p$ by one place to the left
	 	we get a $p$-singular partition also in the block with core $(p-1)$. 
	 	
	 	 Inducing $S_0:= S^{\la p\ra}\downarrow_{B_p}$ to $B$ gives the Young module $Y^{\la p-1\ra}$ which in our notation is $U{\la p\ra \choose \la p-1\ra }$,
	 	Next, inducing  $S_1:= S^{\la p, p-2\ra}\downarrow_{B_p}$ to $B$ contains
	 	$S^{\la p-1, p-2\ra}$ and has quotient $S^{\la p, p-2\ra }$. 
	 	Inducing the exact sequence 
	 	$0\to S_1\to P_0\to S_0\to 0$ in $B_p$ to the block $B$ gives 
	 	$$0\to U{\la p, p-2\ra \choose \la p-1, p-2\ra}\to P^{\la p\ra} \to Y^{\la p-1\ra}\to 0.
	 	$$
	 	where $P^{\la p\ra}$ is the projective cover of $D^{\la p\ra}$,  and we can read off the labels of its Specht subquotients from 
	 	the exact sequence.
	 	Similarly starting with the projective cover of $S_1$, that is $0\to S_2\to P_1\to S_1\to 0$  and inducing to $B$ we obtain
	 	$$0\to U{\la p,  p-3\ra\choose \la p-1, p-3\ra} \to P^{\la p, p-2\ra} \to U{\la p, p-2\ra \choose \la p-1, p-2\ra}\to 0.$$
	 	In this way, we  obtain further Brauer induced modules and indecomposable
	 	projective modules.
	 	
	 	\item \textit{We restrict to the block $B_s$ with core $(s-1, 1^{p-s})$ for $2\leq s < p$. }
	 	
	 	Take the subquiver of the Gabriel quiver of the basic algebra of $k\cS_{2p}$
	 	with the following vertices:
	 	$$\la s\ra, \ \la p, s \ra  \ \la p-1, s\ra, \ldots \la s+1, s\ra, \ \ \la s, s-2\ra, \ \ldots, \ \la s, 1\ra.$$
	 	Consider one of the partitions on the abacus. If we move the bead on runner $u$ one place to the left, then in each case  we get a partition in the block $B_s$, with
	 	core $(s-1, 1^{p-s})$. This gives all $p$-regular partitions in this block, and they form a total order under the dominance order.
	 	Then, similarly to the method explained in (1), we obtain Brauer induced
	 	modules, and also indecomposable projective modules. The module $S^{\la s\ra}\downarrow_{B_s}\uparrow^B$ is the Young module $Y^{\la s-1\ra}$, with 
	 	Specht subquotients $S^{\la s-1\ra}$ and $S^{\la s\ra}$. 
	 	
	 \end{enumerate}
	 
	 	 The information from the Brauer induced modules
	 gives rise to filtrations of indecomposable projective modules and gives most of Theorem \ref{thm:proj}.
	 Whenever a projective module
	 $P^{\la u, v\ra}$ occurs twice by inducing from Brauer tree algebras, it has  two different Brauer induced  quotients, which are then of the form 
	 $U{\la u, v\ra\choose \la u, v-1\ra}$ and $U{\la u, v\ra \choose \la u-1, v\ra}$, the projective module will also contain 
	 $S^{\la u-1, v-1\ra}$, and it has a filtration with quotients
	 $$S^{\la u-1, v-1\ra}, \ \ S^{\la u, v-1\ra}\oplus S^{\la u-1, v\ra}, \ \ S^{\la u, v\ra}.$$
	 This gives the projective modules in parts (c), (e), (f) of Theorem \ref{thm:proj} below. 
	 
	 For parts (a), (d) and (e),  the projective only has one Brauer induced quotient, which we indicate by  the notation $\diagup$.

	 All but one  of the indecomposable projective modules are obtained by inducing from Brauer tree algebras
	 of $k\cS_{2p-1}$,  
	 The remaining module, the one in part (g) of Theorem \ref{thm:proj} below, is obtained as follows.
	 Consider the restriction of  the Specht modules labelled by $\la 1\ra$.
	 We can move one bead from runner $1$ to the left and get the core partition $(p, 1^{p-1})$. This shows that
	 $S^{\la 1\ra}$ restricted to $\cS_{2p-1}$ has  a direct summand which is simple projective. Inducing this
	 to $B$ we get the projective module, and the abacus shows that this has three Specht subquotients, labelled by
	 $\la 1\ra, \ \la p, 1\ra $ and $\la p, p\ra.
	 $

\subsubsection{Indecomposable projective modules of $B$}

Theorem B.4 of \citep{zbMATH07389875}  describes completely the Specht subquotients of the indecomposable projective modules. 
We need to use the submodule structure obtained above using the Brauer induced modules. This is described in the following.

\begin{Theorem}(see \citep[Theorem B.4]{zbMATH07389875})\label{thm:proj} Let $\lambda$ be a $p$-regular partition in $B$. Then $P^{\lambda}$ has the following structure.
	\begin{enumerate}[(a)]
		\item $$P^{\la p \ra} \cong \ \ 
		\begin{matrix} & &S^{\la p \ra }  &  & \cr   
			& \diagup &         && \cr
			S^{\la p-1 \ra }& &&&S^{\la p, p-2\ra} \cr  
			&&& \diagup & \cr
			&&S^{ \la p-1, p-2 \ra }&& \end{matrix};$$
		\item $$P^{\la s \ra} \cong 
		\ \ \begin{matrix}  &S^{\la s \ra }  &   \cr   
			S^{\la s-1 \ra }& \oplus &S^{\la p, s \ra} \cr  
			&S^{ \la p, s-1 \ra }& \end{matrix} \ \quad  (2\leq s\leq p-1);$$
		\item $$P^{\la s,1\ra} \cong \ \ \begin{matrix} & S^{\la s,1 \ra} & \cr  S^{\la s-1, 1 \ra}&\oplus & S^{\la s, s \ra} \cr & S^{\la s-1, s-1 \ra} &\end{matrix}\quad  (  3\leq s\leq p);$$
		\item $$P^{\la s+1, s\ra} \cong \ \ 
		\begin{matrix} & &S^{\la s+1,s  \ra }  &  & \cr   
			& \diagup &         && \cr
			S^{\la s+1, s-1 \ra }& &&&S^{\la s, s-2\ra} \cr  
			&&& \diagup & \cr
			&&S^{ \la s-1, s-2 \ra }&& \end{matrix} \quad  (3\leq s\leq p-1);$$
		\item $$P^{\la 3,2 \ra} \cong \ \   \begin{matrix} &&S^{\la 3,2 \ra } &&\cr
			& \diagup &&&\cr
			S^{\la 3,1\ra } &&&& S^{\la 2,2 \ra} \cr 
			&&& \diagup & \cr
			&&S^{\la 1,1 \ra}&&\end{matrix};$$
		\item $$P^{\la r,s \ra} \cong \ \ \begin{matrix} &S^{\la r,s\ra} & \cr S^{\la r-1,s\ra }&\oplus & S^{\la r, s-1\ra } \cr & S^{\la r-1, s-1\ra}&\end{matrix} \quad   (p\geq r > s > 1 \ \mbox{and} \ r-s>1);$$
		\item $$P^{\la 1\ra} \cong  \ \ \begin{matrix} & S^{\la 1\ra} & \cr 
			& S^{\la p, 1\ra} & \cr &  S^{\la p, p\ra} &\end{matrix}.
		$$
	\end{enumerate}	
\end{Theorem}

In \citep[Corollary B.6.]{zbMATH07389875}, the composition factors of each Specht module are described and consequently the composition factors of all projective indecomposable modules. 
Here, by displaying a module $P$ in the form $\begin{matrix} & &M  &  & \cr   
	& \diagup &         && \cr
	N& &&&X\cr  
	&&& \diagup & \cr
	&&Y&& \end{matrix}$ we mean that $P$ fits into an exact sequence $0\rightarrow \begin{matrix}
	X \\ Y
\end{matrix}\rightarrow P \rightarrow \begin{matrix}
M \\ N
\end{matrix}\rightarrow 0$. This way the Specht filtration of $P$ is the one built from the one of $\begin{matrix}
X \\ Y
\end{matrix} $ at the bottom and the one $\begin{matrix}
M \\ N
\end{matrix}$ in the upper part of the filtration. For modules $M$ displayed as 
$\begin{matrix} &N_1 & \cr N_2&\oplus & N_3\cr & N_4&\end{matrix}$ there exists a filtration with subquotients $N_4$, $N_2\oplus N_3$ and $N_1$. 
In particular, such modules $M$ admit more than one Specht filtration. One such filtration is given by $0=M_5\subset M_4\subset M_3\subset M_2\subset M_1=M$ with $M_i/M_{i+1}\cong N_i$ for $i=1, \ldots, 4$. Moreover, whenever $S^{\lambda}$ occurs in the layer directly below $S^{\mu}$ in the projective module for $\mu$
then there is an indecomposable module containing $S^{\lambda}$ with quotient $S^{\mu}$, which has  simple socle and top, except in the three cases where a line segment $\diagup$ appears. Then for the two quotients with no line segment there is no such module.

These modules of the form $\begin{matrix}
	S^{\mu} \\ S^{\lambda}
\end{matrix}$ are either Brauer induced, or they exist, by the following  lemma.
In particular, the next lemma shows  that the projective modules $P^{\la s \ra}$ can be displayed in the form presented above.

\begin{Lemma} Let $2 \leq s\leq p-1$. 
	Up to isomorphism, there is a unique indecomposable module
	with two Specht subquotients, containing $S^{\la p, s\ra}$ and with 
	quotient $S^{\la s\ra}$. This module has a simple socle and 
	a simple top.
\end{Lemma}

\begin{proof}
By Theorem B.4 of \cite{zbMATH07389875}, the projective module $P^{\la s\ra}$ has a factor module with two Specht subquotients, from the fact that it has a Specht filtration and we know the label
	of its Specht subquotients. Such a module has a simple top.  Similarly, the injective hull of $P^{\la p, s\ra}$ has a submodule with two Specht subquotients and with the same labels, and it has a simple socle.
	We show that ${\rm Ext}_B^1(S^{\la s\ra}, S^{\la p, s\ra}) = k$.
	
		By Remark \ref{rmk2dot4dot4},  ${\rm Hom}_B(S^{\la s \ra}, S^{\la p, s\ra}) = k$ and we see that also  ${\rm Hom}_B(P^{\la s\ra}, S^{\la p, s\ra}) = k$.  
	
By applying ${\rm Hom}_B(-, S^{\la p, s\ra})$ to the exact sequence \ $0\to \Omega^1(S^{\la s\ra}) \to P^{\la s\ra} \to S^{\la s\ra} \to 0$, we deduce that ${\rm Ext}_B^1(S^{\la s\ra}, S^{\la p, s\ra}) \cong {\rm Hom}_B(\Omega^1(S^{\la s\ra}), S^{\la p, s\ra})$. 

The module $\Omega^1(S^{\la s\ra})$ contains $S^{\la p, s-1\ra}$ and $D^{\la p, s-1\ra}$ does not occur in $S^{\la p, s\ra}$. Hence, any homomorphism $\theta$
from $\Omega^1(S^{\la s\ra})$ maps $S^{\la p, s-1\ra}$ to zero and induces a homomorphism $\bar{\theta}$ from the quotient  $\Omega^1(S^{\la s\ra})/S^{\la p, s\ra}$. This  is the direct sum 
$S^{\la p, s\ra}\oplus S^{\la s-1\ra}$. Then $\bar{\theta}$ maps the second summand to zero and it is a multiple of the identity on the first summand.
\end{proof}

 \section{Specht and Young modules over $\Lambda_k(p, 2p)$} \label{SpechtandYoungmodulesoverLambda}

 We now move our focus towards the algebra $\La_k(p, 2p):=\End_{S_k(p, 2p)}(V^{\otimes 2p})^{op}$ and its principal block, where $k$ is a field of characteristic $p\geq 3$.  For the treatment of the case $p=2$, we refer the reader to  \cite{CE24}. As mentioned, this algebra is isomorphic to the quotient algebra $k\cS_{2p}/I_p$ where $I_p$ is the annihilator of the $\cS_{2p}$-action on $V^{\otimes 2p}$.  Using the connection presented in Theorem \ref{thm2dot2dot1}, the simple modules of this algebra are the ones inherited from the Ringel dual of the Schur algebra. Observe that $T(\l)$ is a direct summand of $V^{\otimes 2p}$ if and only if $\lambda$ is a $p$-regular partition of $2p$. It follows that the simple modules of $\Lambda_k(p, 2p)$ are precisely the modules $\top F_{p, 2p}P_R(\l)=\top F_{p, 2p}\St_{R_k(p, 2p)}(\lambda)$ with $\l\in \Lambda^+(p, 2p)$ and $p$-regular. So $\Lambda_k(p, 2p)$ is the  quotient of $k\cS_{2p}$ with a complete set of non-isomorphic simple modules given by 
$\{\top S^\lambda\colon \lambda\in \L^+(p, 2p) \text{ and } \lambda \text{ is $p$-regular} \}$ and with $\{F_{p, 2p}P_R(\l)\colon \lambda\in \L^+(p, 2p) \text{ and } \lambda \text{ is $p$-regular} \}$ as the complete set of projective indecomposable modules.
 
 Let $Q$ be the maximal multiplicity-free direct summand of $V^{\otimes 2p}$ that is a module over the (basic algebra of the) principal block $B$ of $\cS_{2p}$. So $DQ$ as left $B$-module is the direct of all Young modules $Y^\lambda$ whose label satisfies $$\lambda\in (\{\la v\ra\colon 1\leq v\leq p \}\cup \{\la u, v\ra\colon 1\leq v< u \leq p\}\cup \{\la v, v\ra\colon 1\leq v\leq p\}  )\cap \Lambda^+(p, 2p).$$ Then $A_0:=\End_B(Q)$ is the (basic) principal block component of $S_k(p, 2p)$ while $\Lambda_0:=\End_{A_0}(Q)^{op}$ is the (basic) principal block component of $\Lambda_k(p, 2p)$.

 \subsection{Projective $\Lambda_0$-modules} For convenience, we assume now, until the end of the Section \ref{SpechtandYoungmodulesoverLambda}, that $p>3$.
 The projective $\Lambda_0$-modules are obtained from those presented in Theorem \ref{thm:proj} by factoring out the Specht modules where the partition has more than $p$ parts.
  They are precisely the ones labelled as $\la r,   r \ra $ for $1\leq r\leq p$. (Using Subsection \ref{abacusnotation} we see that $\la u\ra$ and $\la u, v\ra$ 
 do not have more than $p$ parts, and $\la u, u\ra$ has  already $p$ beads in the second line of the abacus and has more than $p$ parts).

There are precisely $p$ projective indecomposable modules of $B$ which are not projective for $\Lambda_0$, they are the ones where the label is on the right edge of the Gabriel quiver, that is 
labelled by 
$$\la 1 \ra, \ \la p, 1\ra, \ \ldots, \ \la 3, 1\ra, \ \la 3, 2\ra,
$$
that is, the projective modules described in Parts (c), (e) and (g) of Theorem \ref{thm:proj}.

All other projective modules
of the block remain projective over $\Lambda_0$, and in particular, they are then projective and injective. So, we deduced the following:

\begin{Lemma}\label{lemmaprojectives}
	\begin{enumerate}
		\item The set of non-isomorphic projective indecomposable $\Lambda_0$-modules is labelled by ${\{\la v\ra\colon 1\leq v\leq p\}\cup \{\la u, v\ra\colon 1\leq v<u\leq p\}}\setminus \{\la 2, 1\ra\}$.
		\item Let $\lambda$ be a partition in ${\{\la v\ra\colon 1\leq v\leq p\}\cup \{\la u, v\ra\colon 1\leq v<u\leq p\}}\setminus \{\la 2, 1\ra\}$. Denote by $P_{\Lambda_0}^\lambda$ the projective cover of $\top S^\lambda$. Then, 
		\begin{enumerate}[(i)]
			\item $P_{\Lambda_0}^{\la s\ra}=P^{\la s \ra}$ for $2\leq s\leq p$;
			\item $P_{\Lambda_0}^{\la u, v\ra}=P^{\la u, v\ra}$ whenever $\la u, v\ra \notin \{\la p, 1\ra,  \ldots,  \la 3, 1\ra, \ \la 3, 2\ra\};$
			\item $P_{\Lambda_0}^{\la s, 1\ra} \cong  \begin{matrix}  S^{\la s, 1\ra} & \cr 
				 S^{\la s-1, 1\ra} & \end{matrix}$  for $3\leq s\leq p;$
			\item $P_{\Lambda_0}^{\la 3, 2\ra} \cong   \begin{matrix}  S^{\la 3, 2\ra} & \cr 
				 S^{\la 3, 1\ra} & \end{matrix};$
			\item $P_{\Lambda_0}^{\la 1\ra} \cong   \begin{matrix}  S^{\la 1\ra}  \cr 
				 S^{\la p, 1\ra}  \end{matrix} \ $.
		\end{enumerate}
	\end{enumerate}

The injective indecomposable $\Lambda_0$-modules can then be obtained by applying the simple preserving duality to the projective indecomposable $\Lambda_0$-modules. In particular, if we denote by $I_{\Lambda_0}^{\lambda}$ the injective hull of $\soc S^\lambda$, then $I_{\Lambda_0}^{\la s \ra}=P^{\la s\ra}$ for $2\leq s\leq p$ and $I_{\Lambda_0}^{\la u, v\ra}=P^{\la u, v \ra}$ whenever $\la u, v\ra \notin \{\la p, 1\ra,  \ldots,  \la 3, 1\ra, \ \la 3, 2\ra\}$.

\end{Lemma}

\subsection{Young $\Lambda_0$-modules}

We describe the Young modules for $B$ labelled by partitions with less than or equal to $p$ parts. 

\begin{enumerate}
	\item We start with the projective Young modules. These are precisely those labelled by a partition which is simultaneously $p$-regular and conjugate to a $p$-regular partition of $2p$.
	
	Every indecomposable projective for the block  $B$ is a Young module. Most of these remain projective
	for the algebra $\Lambda_0$, and hence they are Young modules. 
	
	In Theorem \ref{thm:proj},  we have identified which Specht modules are  submodules of an indecomposable projective module. This gives the identification
			of the projective $B$-modules which remain projective for $\Lambda_0$ as Young modules.

\end{enumerate}

	\begin{Cor} Let $\lambda\in \Lambda^+(p, 2p)$ with $p$-weight two. \label{3dot2dot1}
		\begin{enumerate}[(a)]
			\item Let $\lambda \not\in \{ \la a, 1\ra, \ (3\leq a\leq p), \ \la 3, 2\ra, \ \la 1\ra \}$,  and  $\lambda = \la a, b\ra$.
			\begin{enumerate}[(i)]
				\item If $a-b\geq 2$, then
				$P^{\la a, b\ra } \cong Y^{\la a-1, b-1\ra}.$
				\item If   $a-b=1$, then $P^{\la a, b\ra}\cong  Y^{\la a-2, b-2\ra}$.
			\end{enumerate} 
			\item If   $\lambda = \la a \ra$ with $1<a<p$, then  $P^{\la a\ra} \cong Y^{\la p, a-1\ra}$.
			Moreover, $P^{\la p\ra} \cong Y^{\la p-1, p-2\ra}$.
		\end{enumerate}
\end{Cor}
\begin{proof}
	Part (a)(i) follows from Theorem \ref{thm:proj}(f). Part (a)(ii) follows from  Theorem \ref{thm:proj}(d). Part (b) follows from Theorem
	\ref{thm:proj}(b) and the particular case by Theorem  \ref{thm:proj}(a).
\end{proof}

Note $P^{\la 4, 3\ra} \cong Y^{\la 2, 1\ra}$ and $\la 2, 1\ra$ is the only $p$-singular partition in $\Lambda_0$ (see Lemma \ref{singularabacuspartitions}). Hence, $S^{\la 2, 1\ra}$ has composition length $2$, the socle is
$D^{\la 4, 3\ra}$ and the top is $D^{\la 3, 1\ra}$. 

\begin{enumerate}
	\item[(2)] We consider now Young modules which are not projective. 
	
	They are labelled as  $Y^{\la a\ra}$ for $1\leq a\leq p$, together with $Y^{\la p, p-1\ra}$. We have explained in Subsection \ref{restrictingtoBrauertreealgebras} that the Young modules $Y^{\la a\ra}$ for $a< p$ 
	are induced from simple Specht modules of the Brauer tree algebras whose  cores are hook partitions. 
	The socle of $Y^{\la a\ra}$ for $a<p$ is the socle of $S^{\la a\ra}$ which is $D^{\la a+1\ra}$.
	
	The socle of $Y^{\la p, p-1\ra}$ is the socle of the Specht module $S^{\la p, p-1\ra}$, that is, $D^{\la p-1\ra}$.
	
	In terms of their Specht filtrations, this means that we have the following:
	
	\begin{multicols}{2}
			\begin{enumerate}[(i)]
			\item $Y^{\la p \ra}\cong S^{\la p \ra}\cong k$;\\
			\item $Y^{\la p, p-1 \ra}\cong \begin{matrix}  S^{\la p-1\ra}  \cr 
				S^{\la p, p-1\ra}  \end{matrix}$; \\[0.1em]
			\item $Y^{\la a \ra}\cong \begin{matrix}  S^{\la a+1\ra}  \cr 
				S^{\la a\ra}  \end{matrix} \ $ for $1\leq a\leq p-1$.
		\end{enumerate}
	\end{multicols}

\end{enumerate}

\begin{Cor} The indecomposable summands of ${}_{\Lambda_0}DQ$ are
	\begin{enumerate}[(a)]
		\item the trivial module $Y^{\la p\ra}$, 
		\item the projective modules $P^{\la u, v\ra}$ and
		$\la u, v\ra$ not in $\{ \la 1\ra, \la u, 1\ra \  (3\leq u\leq p), \la 3, 2\ra\}$.
		\item the modules $Y^{\la a\ra}$ for $1\leq a < p$, and 
		$Y^{\la p, p-1\ra}$.
	\end{enumerate}
\end{Cor}
\begin{proof}
	The indecomposable direct summands of ${}_{\Lambda_0}DQ$ are the Young modules $Y^\lambda$ with $\lambda\in {\{\la v\ra\colon 1\leq v\leq p\}\cup \{\la u, v\ra\colon 1\leq v<u\leq p\}}$. So, the result follows from Corollary \ref{3dot2dot1}.
\end{proof}

\subsection{Projective Modules as Syzygies of Young Modules}\label{reducing the non injective}

In this part, we are interested in establishing coresolutions of the projective modules of $\Lambda_0$ by modules in $\add DQ$. To do this, we prove that they arise as syzygies of non-projective Young modules.
As a byproduct, we obtain the projective dimensions of all Young $\Lambda_0$-modules.

\begin{Prop}   \label{5dotonedotone}
	We have a bijection between non-injective projective modules of $\Lambda_0$ and the set of  non-projective Young modules  distinct from $Y^{\la p\ra}$:
	
	\begin{enumerate}[(i)]
		\item $\Omega^{p-2}_{\Lambda_0} (Y^{\la a \ra})$ is the non-injective projective module $P_{\Lambda_0}^{ \la a+1, 1\ra}$ for $a=2, 3, \ldots, p-1$.
		\item $\Omega^{p-2}_{\Lambda_0}(Y^{\la 1\ra})$ is the non-injective projective module
		$P_{\Lambda_0 }^{\la 3, 2\ra}$.
		\item $\Omega_{\Lambda_0}^{p-2}(Y^{\la p, p-1\ra})$ is the non-injective
		projective module $P_{\Lambda_0}^{\la 1\ra}$. 
	\end{enumerate}
	Moreover, the following assertions hold.
	\begin{enumerate}[(a)]
		\item The Young module $Y^{\la a\ra}$ and all its syzygies are Brauer induced from the tree with core $(a, 1^{p-a-1})$  for $1\leq a\leq p-1$.
		\item The Young module $Y^{\la p, p-1\ra}$ and its syzygies are  not Brauer induced. 
	\end{enumerate}
\end{Prop}

\begin{proof}
	To describe the resolutions using the Gabriel quiver, begin with the arrow \mbox{$\la a+1\ra \to \la a\ra$} located on the left edge of the quiver. From there, follow parallel arrows to that progressing upward through the quiver until reaching the top. At that point, it changes direction and continues downward to the right edge of the quiver.
	See Part (3) of the proof below.

	We write down the terms in the coresolution of the projective non-injective
	modules. 
	To check that these are correct, one may consult Theorem \ref{thm:proj} and Lemma \ref{lemmaprojectives}.
	
	\begin{enumerate}[(1)]
		\item Consider $P_{\Lambda_0}^{\la 1\ra}$. We have an exact sequence with terms in 
		${\rm add} D(Q_{\Lambda_0})$ 
		$$0 \to P_{\Lambda_0}^{\la 1\ra}  \to P^{\la 2\ra} \to P^{\la 3\ra} \to \cdots \to P^{\la p-1\ra} \to Y^{\la p, p-1\ra} \to 0.
		$$
		\item Consider  $P_{\Lambda_0}^{ \la 3, 2\ra}$. We have
		a coresolution with terms in ${\rm add} D(Q_{\Lambda_0})$
		$$0 \to P_{\Lambda_0}^{\la 3, 2\ra} \to P^{\la 4, 2\ra} \to P^{\la 5, 2\ra} \to \cdots 
		\to P^{\la p, 2\ra} \to P^{\la 2\ra}  \to Y^{\la 1\ra} \to 0.$$
		\item Let $2 <  a \leq p-3$. Then the first part of a coresolution
		for $P_{\Lambda_0}^{\la a+1, 1\ra}$ is 
		$$0 \to P_{\Lambda_0}^{\la a+1, 1\ra} \to P^{\la a+1, 2\ra} \to \cdots \to P^{\la a+1, a-1\ra} $$
		and the last quotient is $U{\la a+1, a-1\choose \la a, a-1\ra}$.
		This has injective hull $P^{\la a+2, a+1\ra}$. Hence, the coresolution continues 
		$$\cdots \to P^{\la a+2, a+1\ra} \to P^{\la a+3, a+1\ra} \to \cdots \to P^{\la p, a+1\ra} \to P^{\la a+1 \ra}
		\to Y^{\la a\ra} \to 0.
		$$
		\item Consider $a=2$. 
		Then we have the coresolution
		$$0\to P_{\Lambda_0}^{\la 3, 1\ra} \to P^{\la 4, 3\ra} \to P^{\la 5, 3\ra} \to \cdots \to
		P^{\la p, 3\ra} \to P^{\la 3\ra} \to Y^{\la 2\ra} \to 0.
		$$
		\item This leaves two coresolutions, they are
		
		$$0 \to P_{\Lambda_0}^{ \la p-1, 1\ra} \to P^{\la p-1, 2\ra} \to \cdots \to  P^{\la p-1, p-3\ra} 
		\to P^{\la p, p-1\ra}  \to P^{\la p-1\ra} \to Y^{\la p-2\ra} \to 0.$$
		and 
		$$0\to P_{\Lambda_0}^{\la p, 1\ra} \to P^{\la p, 2\ra} \to \cdots \to  P^{\la p, p-2\ra} \to P^{\la p\ra} \to Y^{\la p-1\ra} \to 0.$$
	\end{enumerate}
\end{proof}

\begin{Remark} \label{pdimYoungmodules}
	We label the terms of the coresolution starting with $0$, and we see that
	the Young module occurs in step $p-2$. Hence, all non-projective Young modules
	other than $Y^{\la p\ra}$ have projective dimension $p-2$.
\end{Remark}

Surprisingly, Proposition \ref{5dotonedotone} gives a completely different picture about the homological structure of $\Lambda_0$ in contrast with its counterpart in characteristic two, the principal block of $TL_{k, 4}(0)$.
 The Temperley-Lieb algebra $TL_{k, 4}(0)$ studied in \cite{CE24} is very far from being Iwanaga-Gorenstein. On the other hand, we can use Proposition \ref{5dotonedotone} to show that the algebra $\Lambda_0$ is Iwanaga-Gorenstein of infinite global dimension. 

\begin{Theorem}Let $k$ be an algebraically closed field with characteristic $p>3$. Then,
	 $\Lambda_k(p, 2p)$ and its principal block $\Lambda_0$ are Iwanaga-Gorenstein algebras of infinite global dimension. \label{thm3:3:2}
\end{Theorem}
\begin{proof}
	We will start by showing that $\Lambda_0$ is Iwanaga-Gorenstein. By Proposition \ref{5dotonedotone}, all the non-projective Young modules except $Y^{\la p\ra}$ have finite projective dimension over $\Lambda_0$. Observe that the simple preserving duality interchanges projective with injective modules. Since the Young modules are self-dual with respect to the simple preserving duality functor, all the Young modules, except $Y^{\la p\ra}$, have finite injective dimension. Observe that the middle terms in the exact sequences (1)-(5) of the proof of Proposition \ref{5dotonedotone} are projective-injective by Lemma \ref{lemmaprojectives} (since they are the restriction from a projective module over a self-injective algebra in the quotient algebra). Thus, it follows from the exact sequences (1)-(5) of the proof of Proposition \ref{lemmaprojectives} that the projective modules $P_{\Lambda_0}^{\la 1\ra}$, $P_{\Lambda_0 }^{\la 3, 2\ra}$ and $P_{\Lambda_0}^{ \la a+1, 1\ra}$, for $a=2, 3, \ldots, p-1$, have finite injective dimension. Since all the other projective indecomposable $\Lambda_0$-modules are injective (since they are, in particular, Young modules) we obtain that $\Lambda_0$ has finite injective dimension. Using the simple preserving duality, we obtain that all injective indecomposable modules have finite projective dimension, and therefore $\Lambda_0$ is Iwanaga-Gorenstein. 
	
	To see that $\Lambda_0$ has infinite global dimension  observe that $\Lambda_0$ and $\Lambda_k(p, 2p)$ are cellular algebras (see for instance \citep[Remark 8.1.4]{Cr2} or \citep[Theorem 7.5]{appendix} together with the fact that we can write $\End_{R_k(p, 2p)}(\Hom_{S_k(p, 2p)}(T, V^{\otimes 2p}) )^{op}\cong \Lambda_k(p, 2p)$). By Equation (2), the cell modules in the cellular structure of these algebras are precisely the Specht modules.
	
	Then, by \cite{zbMATH01384521}, $\Lambda_k(p, 2p)$ and, in particular, its principal block $\Lambda_0$ would be quasi-hereditary if they had finite global dimension. But they cannot be quasi-hereditary since the number of Specht modules (cell modules) differs from the number of simple modules by one.
	
	The module $V^{\otimes 2p}\oplus T(2^p)$ is (up to multiplicities) a characteristic tilting module of $S_k(p, 2p)$ and $T(2^p)$ is in the principal block of $S_k(p, 2p)$. So the direct summand of $V^{\otimes 2p}$ in a non-principal block component of $S_k(p, 2p)$ is a characteristic tilting module in the block component. Hence, 
	the non-principal blocks of $\Lambda_k(p, 2p)=\End_{S_k(p, 2p)}(V^{\otimes 2p})^{op}$ are quasi-hereditary algebras (actually they are the Ringel dual of the respective block component of $S_k(p, 2p)$), see also  \citep[4.6(3)]{E1}. In particular, the non-principal blocks of $\Lambda_k(p, 2p)$ have finite global dimension, and so $\Lambda_k(p, 2p)$ is an Iwanaga-Gorenstein algebra.
\end{proof}

	\begin{Cor}\label{finitisticlamdbazero}
	The algebra $\Lambda_0$ has finitistic dimension equal to $2p-4$.
\end{Cor}
\begin{proof}Let $P$ be a projective  not injective $\Lambda_0$-module.
	By Proposition \ref{5dotonedotone}, it follows that $\injdim_{\Lambda_0}P=p-2+\injdim Y^\lambda$ with $\lambda\in \{\la p, p-1 \ra\}\cup \{\la a\ra \colon 1\leq a \leq p-1\}$. As we have seen, in such a case, $\injdim_{\Lambda_0} Y^\lambda=\pdim_{\Lambda_0} Y^\lambda=p-2$. Since $\Lambda_0$ is Iwanaga-Gorenstein the finitistic dimension of $\Lambda_0$ is precisely $\injdim_{\Lambda_0} \Lambda_0$ (see \citep[Proposition 4.2]{zbMATH00067390} and \citep[Theorem 2.28]{zbMATH02078857}). Hence, it is equal to $p-2+p-2=2p-4$. 
\end{proof}

Later, we will return to the  Young module not covered in Proposition \ref{5dotonedotone} to see that it has infinite projective dimension.

\subsection{Coresolving $S^{\la 2, 1 \ra}$ by $DQ$} \label{coresolvingsingularSpechtmodule}

The main motivation to resolve the singular Specht module $S^{\la 2, 1 \ra}$ by $DQ$ is two-fold:
\begin{enumerate}[(i)]
	\item Determine the global dimension of the principal block $A_0$ from the point of view of \citep[Lemma 3.8]{CP2};
	\item Determine the relative dominant dimension $Q\domdim_{A_0} T(2^p)$, where $T(2^p)\oplus Q$ is the (multiplicity-free) characteristic tilting module of $A_0$. 
\end{enumerate}

Our aim is to
construct an exact sequence of the form
\begin{align}
	0 \to S^{\la 2, 1\ra} \to R_0\stackrel{d_1}\to R_1 \to \ldots \to R_{r-1} \stackrel{d_r}\to R_r \to \ldots \to R_{2(p-1)}\to 0, \label{eq3}
\end{align}
with terms in ${\rm add}_{\Lambda_0}DQ$, that is, each $R_i$ is a direct sum of Young modules.
We let  $Z_r = {\rm im} (d_r)$ for $r=1, \ldots, 2p-2$.  The coresolution will be computed inductively and the main idea is to describe
{\it the kernel $Z_r$ as an extension of a distinguished Specht module
	by a module $W_r$, where $W_r$ is an  (explicit) iterated extension of 
	Brauer induced modules.}
Here, by a \emph{distinguished Specht module} we mean a Specht module of the form $S^{\la u, v\ra }$ with $u-v\in \{1, 2\}$.

If $r$ is odd, we write $r=2a-1$. Then the Specht module which will
be the distinguished submodule of $Z_r$ is $S^{\la a+2, a\ra}$.
Otherwise, let $r=2a$. Then the distinguished submodule of
$Z_r$ will be $S^{\la a+2, a+1\ra }$.

\begin{Example} We describe the start of the construction
	in detail, and it will motivate
	the general step. \label{example3dot4dot1}
	
	\begin{enumerate}[(I)]
		\item As the start, we
		take $R_0$ to be the injective hull of $S^{\la 2, 1\ra}$, and then fix
		$Z_1= \Omega^{-1}S^{\la 2, 1 \ra}$.
		
		By Lemma \ref{lemmaprojectives},  the injective hull of $S^{\la 2, 1\ra}$ is
		the Young module $Y^{\la 2, 1\ra}$ which is isomorphic to 
		$P^{\la 4, 3\ra}$.
		We have the exact sequence
		\begin{align}
			0\to U{\la 3, 1 \ra\choose \la 2, 1\ra} \longrightarrow
			Y^{\la 2, 1\ra} \longrightarrow   U{\la 4,3\ra \choose \la 4, 2\ra}
			\to 0. \label{eq7}
		\end{align}
		
		So $R_0= Y^{\la 2, 1\ra}$
		and we define $Z_1$ to be the quotient $R_0/S^{\la 2, 1\ra}$. 
		
		The term on the right of (\ref{eq7}) is Brauer induced. We claim that it is 
		isomorphic to $\Omega^{p-3}Y^{\la 2\ra}$. This follows since the
		term on the left of (\ref{eq7}) is, 
		by Proposition \ref{5dotonedotone}, isomorphic to $\Omega^{p-2}Y^{\la 2\ra}$. 
		
		This gives a short exact sequence
		\begin{align}
			0\to S^{\la 3, 1\ra} \to Z_1 \to  \Omega^{p-3}Y^{\la 2\ra}=:W_1\to 0. \label{eq8}
		\end{align}

		\item We take injective hulls of the end terms of the sequence (\ref{eq8}), and denote their direct sum by $R_1$. Then 
		$Z_1$ embeds into $R_1$ and 
		we define $Z_2$ be the quotient $R_1/Z_1$. 
		The injective hull of $S^{\la 3, 1\ra}$ is the Young module $Y^{\la 3, 1\ra}$ (which is isomorphic to $P^{\la 4, 2\ra}$). 
		By Snake Lemma and (4) of the proof of Proposition \ref{5dotonedotone}, $Z_2$ is then given by the  exact sequence
		\begin{align}
			0\to P^{\la 4, 2 \ra}/S^{\la 3, 1\ra} \to Z_2 \to \Omega^{p-4}Y^{\la 2\ra}\to 0. \label{eq11}
		\end{align}
			The module  $P^{\la 4, 2\ra}$ has a filtration by Brauer
		induced modules coming  from
		the projective resolution of $Y^{\la 1\ra}$:
		\begin{align}
			0\to U{\la 3,2\ra \choose \la 3, 1\ra} \to P^{\la 4, 2\ra} \to U{\la 4,2\ra \choose \la 4, 1\ra} \to 0. \label{eq9}
		\end{align}
	The right term of the sequence (\ref{eq9}) is isomorphic to $\Omega^{p-3}Y^{\la 1\ra}$. Factoring out $S^{\la 3, 1\ra}$ from (\ref{eq9}) gives the exact sequence 
	\begin{align}
	0\to S^{\la 3, 2\ra} \to P^{\la 4,2\ra}/S^{\la 3, 1\ra}  \to \Omega^{p-3}Y^{\la 1\ra} \to 0. \label{eq10}
	\end{align}
By Equations (\ref{eq11}) and (\ref{eq10}), we have $S^{\la 3, 2\ra} \subset Z_2$ and the quotient $W_2$
is an extension
$$0\to \Omega^{p-3}Y^{\la 1 \ra} \to W_2 \to \Omega^{p-4}Y^{\la 2\ra}\to 0.
$$ 
	\end{enumerate}
\end{Example}

In general, the method is the same. Let $\lambda$ be a partition in $\L^+(p, 2p)$.

\begin{enumerate}[(I)]
	\item We take the injective
	hull of the distinguished Specht modules, and we write the quotient as
	an extension of a syzygy of some $Y^ \lambda$ by another distinguished 
	Specht module. 
	\item We combine the syzygy of $Y^ \lambda$ with the iterated
	extension which we get from the previous step.
\end{enumerate}

\subsubsection{Description of the cosygyzy of the distinguished Specht modules}

We will use the following lemma as a tool to realise the procedure (I).

\begin{Lemma}  \label{lemma3dot4dot2} Let $1\leq a \leq p-2$. Then there are Brauer induced exact sequences
	\begin{align}
		0\to U{\la a+2, a\ra \choose \la a+1, a\ra} \longrightarrow Y^{\la a+1, a\ra} \longrightarrow \Omega^{p- (a+2)}(Y^{\la a+1\ra})\to 0 \\
		0\to U{\la a+2, a+1\ra \choose \la a+2, a\ra} \longrightarrow Y^{\la a+2, a\ra} \to \Omega^{p-(a+2)}(Y^{\la a\ra})\to 0.
	\end{align}
\end{Lemma}
\begin{proof}
Consider the first part.	By Corollary \ref{3dot2dot1}, we have

$$Y^{\la a+1, a\ra} = \left\{\begin{array}{ll} P^{\la a+3, a+2\ra} & \text{ if } a\leq p-3
	\cr P^{\la p\ra} & \text{ if } a=p-2
\end{array}\right..
$$ By Theorem \ref{thm:proj}, this has (just one) filtration by Brauer induced modules:
$$0 \to U{\la a+2, a\ra\choose \la a+1, a\ra} \to Y^{\la a+1, a\ra} \to 
U_a \to 0,
$$ where $$U_a=\begin{cases}
	U{\la a+3, a+2\ra \choose \la a+3, a+1\ra}, &\text{ if } a\leq p-3 \\[1em]
	U{\la p\ra \choose \la p-1\ra}, &\text{ if } a=p-2
\end{cases}.$$
If $a=p-2$, then $U_a$ is $Y^{\la p-1\ra}$ (see Subsection \ref{restrictingtoBrauertreealgebras}(1)).

Suppose $a\leq p-3$. Then $U_a$ is $\Omega^{p-(a+2)}Y^{\la a+1\ra}$.
This follows from Proposition \ref{5dotonedotone} by observing that   the first syzygies of 
$Y^{\la a+1\ra}$ are
$U{\la u, a+1\ra\choose \la u, a\ra}$ for $u=p, p-1, \ldots, a+2$. 

Now consider the second part. By Corollary \ref{3dot2dot1},
we have $Y^{\la a+2, a\ra} =\begin{cases}
	P^{\la a+3, a+1\ra}, &\text{ if } a\leq p-3 \\
	P^{\la p-1 \ra}, &\text{ if } a=p-2
\end{cases} $. 
If $a=p-2$, Theorem  \ref{thm:proj}(b) yields the filtration
\begin{align}
	0\rightarrow U{\la p, p-1\ra \choose \la p, p-2\ra} \to Y^{\la p, p-2\ra}
	\to U{\la p-1\ra \choose \la p-2\ra}\to 0. \label{eq13}
\end{align}
Otherwise, using Theorem \ref{thm:proj}(f) we take the filtration by Brauer induced modules 
\begin{align}
	0 \to U{\la a+2, a+1\ra \choose \la a+2, a\ra} \to Y^{\la a+2, a\ra}
	\to U{\la a+3, a+1\ra \choose \la a+3, a\ra}\to 0, \label{eq12}
\end{align}
In both instances, in (\ref{eq11}) and (\ref{eq12}), the module on  the right is $\Omega^{p-(a+2)}Y^{\la a\ra}$, again by Proposition \ref{5dotonedotone} (see Subsection \ref{restrictingtoBrauertreealgebras}(2)).
\end{proof}

The following shows that, in characteristics greater than 3, filtrations by Young modules simply correspond to taking direct sums.

	\begin{Lemma}Assume that $p>3$. \label{lemmafiltrationsYoung}
	Let $\lambda$ and $\mu$ be partitions of $2p$ in at most $p$ parts. Then, 
	$\Ext_{\Lambda_k(p, 2p)}^1(Y^\lambda, Y^\mu)=0$. In particular, $\Ext_{\Lambda_0}^1(Y^\lambda, Y^\mu)=0$ for any two partitions in $\Lambda_0$.
\end{Lemma}
\begin{proof}
	By Equation (\ref{eqone}), we can write $\Ext_{\Lambda_k(p, 2p)}^1(Y^\lambda, Y^\mu)\cong \Ext_{S_k(p, 2p)}^1(LI(\lambda), LI(\mu))$. By Theorem \ref{thm2dot2dot1}(c), we obtain $\Ext_{\Lambda_k(p, 2p)}^1(LI(\lambda), LI(\mu))\cong \Ext^1_{S_k(p, 2p)}(I(\lambda), I(\mu))=0$.
\end{proof}

\subsubsection{Description of $W_r$ and $Z_r$}

The task is to construct $Z_r$ with an  embedding of $Z_r$ into a module in ${\rm add} DQ$. 
We will embed it into a module $I\oplus Y$ where $I$ is injective, and $Y$ is a direct sum of non-projective Young modules.

We will get that  each  module $W= W_r$ occurs as   an exact sequence
\begin{equation}
	0\to W'\to W\to W''\to 0 \label{eqextension}
\end{equation}
such that  $W'$  is  filtered by modules of the form 
$\Omega^{p-u}Y^{\la v\ra}$ for some $\lceil r/2 \rceil+2\leq u < p$ and for some $v$ with $1\leq v\leq p-1$, and $W''$ is zero or is filtered by  
non-projective  Young modules.
 Then $W''$ is a direct sum of these modules, by Lemma \ref{lemmafiltrationsYoung}. 
Let $I(W')$ be the injective hull of $W'$. 
Then we can embed $W$ into the direct sum of $I(W')\oplus  W''$,  and this belongs to ${\rm add} DQ$.

Each module $Z_r$ can be described as a Specht module glued onto an iterated extension of Brauer-induced modules, as follows:

\begin{Prop} Let $r$ be a number in $\{1, \ldots, 2p-4\}$. Set $W_0=0$. \label{prop3dot4dot3}
	\begin{enumerate}[(a)]
		\item Suppose $r=2a-1\geq 1$ and $a+2\leq p$.  Then we have an exact sequence
		$$0\to S^{\la a+2, a\ra} \to Z_{r} \to  W_{r}\to 0
		$$
		where $W_r$ is an extension
		$$0 \to \Omega^{p-(a+2)}Y^{\la a+1\ra} \to W_r \to \Omega^{-1}W_{r-1}'\to 0,$$ that can be rearranged into an extension of the form (\ref{eqextension}), and $Z_r$ can be  embedded into a module in ${\rm add} DQ$.
		\item Suppose $r=2a\geq 1$ and $a+2 \leq p$. Then we have an exact sequence
		$$0\to S^{\la a+2, a+1\ra} \to Z_r \to W_r \to 0$$
		where $W_r$ is an extension
		$$0\to \Omega^{p-(a+2)}Y^{\la a \ra} \to W_r \to \Omega^{-1}W_{r-1}' \to 0,$$ that can be rearranged into an extension of the form (\ref{eqextension}), and $Z_r$ can be  embedded into a module in ${\rm add} DQ$.
	\end{enumerate}
\end{Prop}
\begin{proof}
	We prove the claim by induction on $r$. In the
	Example \ref{example3dot4dot1}, we have dealt with $r=1$ and $r=2$.
	\begin{enumerate}[(a)]
		\item Assume $r=2a-1\geq 1$ and $a< p-2$. We will construct the relevant exact sequence for $2a$. By induction, we have an exact sequence
		$$0\to S^{\la a+2, a\ra} \to Z_r \to W_r\to 0$$
		and $W_r$ is an iterated extension of modules of the form $\Omega^{p-u}Y^{\la v\ra}$ for some $a+2\leq u\leq p$ and $1\leq v\leq p-1$. 
		
		We embed the term on the left into $Y^{\la a+2, a\ra}$, and the quotient is described in Lemma \ref{lemma3dot4dot2}. 
		By induction, $W_r$ comes as an exact sequence
		$0\to W_r'\to W_r \to W_r''\to 0$, with the properties described in (\ref{eqextension}).
		
		We embed each of the quotients $\Omega^{p-u}Y^{\la v\ra}$ of $W_r'$  (so that  $p-u\geq 1$) into its injective hull. For $W_r''$ we take the identity map. Then $W_r$ can be embedded into the direct sum of these modules, which we denote
		by $R$ (and it belongs to ${\rm add} DQ$). The correspoding quotient is $\Omega^{-1}W_r'$ and thus 
		 an iterated extension of modules of the form $\Omega^{p-u-1}Y^{\la v\ra}$ for some $a+2\leq u\leq p$ and $1\leq v\leq p-1$.
		Then $Z_r$ can be embedded into the direct sum  of $Y^{\la a+2, a\ra}$ with $R$,  and we take the quotient to be $Z_{r+1} = Z_{2a}$. That is, by the Snake Lemma, we constructed an exact sequence
		$$0 \to Y^{\la a+2, a\ra}/S^{\la a+2, a\ra} \to Z_{2a} \to {\Omega}^{-1}W_{2a-1}'\to 0.$$
		By Lemma \ref{lemma3dot4dot2}, $X:= Y^{\la a+2, a\ra} /S^{\la a+2, a\ra}$ is an extension of $\Omega^{p-(a+2)}Y^{\la a\ra}$ by $S^{\la a+2, a+1\ra}$. 
		We define $W_{2a}$ to be the quotient $Z_{2a}/S^{\la a+2, a+1\ra}$, this gives an exact sequence
		$$0\to S^{\la a+2, a +1\ra}\to Z_{2a}\to W_{2a}\to 0$$
		together with the following identification:
		\begin{align*}
			{\Omega}^{-1}W_{2a-1}'\cong Z_{2a}/X\cong (Z_{2a}/S^{\la a+2, a+1\ra})/(X/S^{\la a+2, a+1\ra}) \cong (Z_{2a}/S^{\la a+2, a+1\ra})/\Omega^{p-(a+2)}Y^{\la a\ra}.
		\end{align*}
		So $W_{2a}$ fits into an extension of the correct form. In particular, combining the iterated extension of modules of ${\Omega}^{-1}W_{2a-1}'$ with $\Omega^{p-(a+2)}Y^{\la a\ra}$, we get that  $W_{2a}$ fits into an exact sequence of the form (\ref{eqextension}).

		 \item  Now assume that $r=2a$ (and  $a< p-2$). In this case, we have by the inductive hypothesis  an exact sequence
		 $$0\to S^{\la a+2, a+1\ra} \to Z_r \to W_r\to 0$$
		 and $W_r$ is an iterated extension of modules of the form $\Omega^{p-u}Y^{\la v\ra}$ for some $a \leq u\leq p$ and $1\leq v\leq p-1$. 
		 
		 We embed the term on the left into $Y^{\la a+2, a+1\ra}$, and the quotient is described in Lemma \ref{lemma3dot4dot2}.
		 By induction, $W_r$ comes as an exact sequence
		 $0\to W_r'\to W_r \to W_r''\to 0$, with the properties described in (\ref{eqextension}).
		 
		 We embed each of the quotients $\Omega^{p-u}Y^{\la v\ra}$ of $W_r'$  (so that  $p-u\geq 1$) into its injective hull. For $W_r''$ we take the identity map. Then $W_r$ can be embedded into the direct sum of these modules, which we denote
		 by $R$ (and it belongs to ${\rm add} DQ$). The correspoding quotient is isomorphic to $\Omega^{-1}W_r'$.

		 
		 Then $Z_r$ can be embedded into the direct sum  of $Y^{\la a+2, a+1 \ra}$ with $R$,  and we take the quotient to be $Z_{r+1}$. 
		 As the final step, we rearrange as in (a). So we get an exact sequence
		 $$0\to S^{\la a+3, a +1\ra}\to Z_{r+1 }\to W_{r+1}\to 0$$
		 where $W_{r+1}$ is an extension of ${\Omega}^{-1}W_r'$ by the Brauer induced module $\Omega^{p-(a+3)}Y^{\la a+2\ra}$.  This sequence is of the form as stated. 
	\end{enumerate}
\end{proof}

\subsubsection{The last step in the coresolution} 

Now, if we use Proposition \ref{prop3dot4dot3} to build the exact sequence (\ref{eq3}), all the middle terms $R_i$ in the coresolution that are injective Young modules can be determined. Moreover, at this stage, the construction exhausts all possible injective Young modules. Nevertheless, Proposition \ref{prop3dot4dot3} does not suffice to complete the sequence, as the final term $Z_r$ produced by the proposition is not a Young module.

The  last term in the above construction occurs when $a+2=p$ and $r=2a=2p-4$. In that case, (b) gives the exact sequence
\begin{align}
	0\to S^{\la p, p-1\ra} \to Z_{2p-4}\to W_{2p-4}\to 0 \label{eq14}
\end{align}
such that $W_{2p-4}$ fits into an exact sequence of the following form
$$0\to Y^{\la p-2\ra}\to W_{2p-4}\to \Omega^{-1}W_{2p-5}'\to 0.$$
By Proposition \ref{prop3dot4dot3}(a) for $r=2p-5=2(p-2)-1$, we have that $W_{2p-5}$ is an iterated extension with quotients being Young modules. So, we deduce that $W_{2p-5}'=0$. So (\ref{eq14}) can be simplified into
\begin{align}
	0\to S^{\la p, p-1\ra} \to Z_{2p-4} \to Y^{\la p-2\ra}\to 0. \label{eq15}
\end{align}

	We embed the term on the left of (\ref{eq15}) into $Y^{\la p, p-1\ra}$ and
	we take the identity for the term on the right of (\ref{eq15}), hence we take $R_{2p-4}=Y^{\la p, p-1\ra}\oplus Y^{\la p-2 \ra}$.
	The Young module $Y^{\la p, p-1\ra}$ is not projective but it is still true 
	that 
	$Z_{2p-4}$ can be  embedded into $R_{2p-4}$ since $\Ext_{\Lambda_0}^1(Y^{\la p-2 \ra}, Y^{\la p, p-1\ra})=0$.

We let $Z_{2p-3}$ be the quotient $R_{2p-4}/Z_{2p-4}$, which is 
$Y^{\la p, p-1\ra}/S^{\la p, p-1\ra} \cong
S^{\la p-1\ra}$.
We embed this into the Young module $R_{2p-3}:=Y^{\la p-1\ra}$ and we obtain an exact sequence
$$0\to  S^{\la p-1\ra} \to Y^{\la p-1\ra} \to Y^{\la p\ra} \to 0.$$
Thus, the next term $Z_{2p-2}$ is  $Y^{\la p\ra}$ and so the coresolution ends with $R_{2p-2}=Z_{2p-2}=Y^{\la p\ra}$.

	\section{Cohomological properties of Young and Specht modules} \label{Cohomological properties of Young and Specht modules}
	
	In this section, we will make use of the resolutions built in Proposition \ref{5dotonedotone} to give two properties about $Q$: one as a module over the Schur algebra, another as a module over $\Lambda_k(p, 2p)$. Namely, we will determine the relative dominant dimension of the Schur algebra $S_k(p, 2p)$ with respect to the tensor power $V^{\otimes 2p}$ and prove that tensor power as $\Lambda_k(p, 2p)$-module contains a full tilting $\Lambda_k(p, 2p)$-module as direct summand. To keep the notation short, we will abbreviate $\Lambda_k(p, 2p)$ to $\Lambda$ whenever possible. To avoid confusion, we will write $Q_{\Lambda_0}=Q_\Lambda$ when we view $Q$ as right $\Lambda$-module while we will write $Q$ when we view it as left $S_k(p, 2p)$-module (or as left $A_0$-module).
	
		\subsection{The relative dominant dimension of $S_k(p, 2p)$ with respect to the tensor power}

The strategy is to determine it through the relative dominant dimension of the characteristic tilting module using the relative Mueller's theorem established in \citep[Theorem 3.1.4]{Cr2}.

\begin{Lemma}\label{4dot1dot1}
	$D\Hom_{A_0}(Q, T(2^p))\cong DS^{\langle 2, 1\rangle }$ as right $\Lambda_0$-modules.
\end{Lemma}
\begin{proof}
	In abacus labelling, $2^p$ corresponds to $\langle 2, 1\rangle$. Since $T(2^p)\cong \Cs(2^p)$, the result follows from \citep[Proposition 5.2]{E1}.
\end{proof}

To keep working with left $\Lambda_k(p, 2p)$-modules, we want to show that ${\rm Ext}^i_{\Lambda_0}(D(Q_\Lambda), X)=0$ for $1\leq i\leq t-2$ where $t= 2(p-1)$, for $X$ the dual Specht module ${}^\natural S^{\la 2, 1 \ra}$. Since $\Lambda_0$ is a block of $\Lambda$ we have $\Ext_{\Lambda}^i(M, N)\cong \Ext^i_{\Lambda_0}(M, N)$ for all $\Lambda_0$-modules $M$ and $N$. We know that
$X$ has length two, socle isomorphic to $D^{\la 3, 1\ra}$ and top isomorphic to $D^{\la 4, 3\ra}$. 
We only need to determine ${\rm Ext}^i_{\Lambda} (Y^{\lambda}, X)$ when $Y^{\lambda}$ is not projective.

\begin{Lemma} \label{lemma4dot2} Assume that $p>3$. Then,
	${\rm Ext}^i_{\Lambda_k(p, 2p)}(Y^{\lambda}, X)=0$ for all $i\geq 1$ and $\lambda\neq \la p\ra$.
\end{Lemma}
\begin{proof} If $Y^{\lambda}$ is projective, then the result is clear. Assume that $Y^{\lambda}$ is not projective and $\lambda \neq \la p\ra$. 
	Since $Y^{\lambda}$ is self-dual we obtain by Theorem \ref{thm2dot2dot1} 
	$$\Ext_{\Lambda}^i(Y^\lambda, {}^\natural S^{\la 2, 1 \ra})\cong \Ext_{\Lambda}^i(S^{\la 2, 1\ra}, Y^\lambda)\cong \Ext_{S_k(p, 2p)}^i(\Cs(2^p), I(\lambda))=0$$ for all $1\leq i\leq p-3$.

	By Remark \ref{pdimYoungmodules}, $Y^{\lambda}$ has projective dimension equal to $p-2$ and the middle terms of the resolution constructed are projective and injective.
	
	 So, $\Ext_{\Lambda}^i(Y^\lambda, X)=0$ for $i>p-2$ and $${\rm Ext}^{p-2}_{\Lambda}(Y^{ \lambda}, X) \cong {\rm Ext}^1_{\Lambda}(\Omega^{p-3}(Y^{ \lambda}), X)\cong \overline{\Hom}_{\Lambda}(\Omega^{p-2}(Y^\lambda), X).$$
		So we only need to deal with a resolution where non-zero homomorphisms occur. Since $X$ has length two and the socle is labelled by $\la 3, 1\ra$, non-zero homomorphisms only occur when $\lambda=\la 2\ra$ by Proposition \ref{5dotonedotone}. So, consider the exact sequence
		\begin{align}
			0\to \Omega^{p-2}(Y^{\la 2\ra}) \cong P_{\Lambda_0}^{\la 3, 1\ra} \  \to P^{\la 4, 3\ra} \to \Omega^{p-3}(Y^{\la 2\ra})\to 0 \label{eq15double}
		\end{align}
where 	$\Omega^{p-3}(Y^{\la 2\ra})  =U \cong  U{\la 4, 3\ra \choose \la 4, 2\ra}$.

	Applying ${\rm Hom}_{\Lambda}(-, X)$ to the exact sequence (\ref{eq15double}) we get the exact sequence
	$$0\to {\rm Hom}_{\Lambda}(U, X) \to  {\rm Hom}_{\Lambda} (P^{\la 4, 3\ra}, X) \to  {\rm Hom}_{\Lambda} (P_{\Lambda_0}^{\la 3, 1\ra}, X) \to 
	{\rm Ext}^1_{\Lambda}(U, X) \to 0.
	$$
	
	The second and the third terms of the sequence are each 1-dimensional and so 
	${\rm Ext}^1_{\Lambda}(U, X)$ and ${\rm Hom}_{\Lambda}(U, X)$ have the same vector space dimension.
	So, it is enough to show that ${\rm Hom}_{\Lambda}(U, X)$ is zero.  We have
	$$0\to S^{\la 4,2\ra} \to U  \to S^{\la 4, 3\ra} \to 0$$
	and ${\rm Hom}_{\Lambda}(S^{\la 4, 2\ra},  X) = 0$ since $D^{\la 4, 2\ra}$ is not a composition factor of $X$. Moreover $S^{\la 4, 3\ra}$ is uniserial with composition factors $D^{\la 4, 3\ra}, \ D^{\la 5, 3\ra}$ and $D^{\la 6, 5\ra}$ and therefore also
	${\rm Hom}_{\Lambda}(S^{\la 4,  3\ra}, X)= 0$.
\end{proof}

\begin{Lemma}
	Let $r$ be a number in $\{1, \ldots, 2p-4\}$. Then, the following assertions hold.
	\begin{enumerate}[(a)]
		\item $\Ext_\Lambda^i(W_r, X)=0$ for every $i\geq 0$, where $W_r$ is the module defined in Proposition \ref{prop3dot4dot3}.
		\item $\Hom_\Lambda(Z_r, X)=\begin{cases}
			k, \text{ if } r=1\\
			0, \text{ if } r\in\{2, \ldots, 2p-4\}
		\end{cases}$, where $Z_r$ is the kernel defined in the coresolution (\ref{eq3}).
	\end{enumerate}  \label{lemma4dotonedot3}
\end{Lemma}
\begin{proof}
	The module $W_r$ is an iterated extension with quotients being syzygies of $Y^\lambda$ with $\lambda\in \{\la 2\ra, \ldots, \la p-1\ra\}$. By Lemma \ref{lemma4dot2}, $\Ext_\Lambda^i(\Omega^{p-u}Y^{\la v\ra}, X)\cong \Ext_\Lambda^{i+p-u}(Y^{\la v \ra}, X)=0$ for every $i\geq 1$. Thus, $\Ext_\Lambda^i(W_r, X)=0$ for every $i\geq 1$. So it remains to check that $\Hom_\Lambda(W_r, X)=0$. To do this, observe that the only Brauer induced module which has non-zero homomorphisms to $X$ is $P_{\Lambda_0}^{\la 3, 1\ra}\cong \Omega^{p-2}(Y^{\la 2 \ra})$. Indeed, given the composition factors of $X$, we only need to check Brauer induced modules which have top quotient $S^{\la 4, 3\ra}$ or $S^{\la 3, 1\ra}$. The Specht module $S^{\la 4, 3\ra}$ does not have $D^{\la 3, 1\ra}$ as a composition factor, so there is no non-zero homomorphism in such a case. Of course, there is a non-zero homomorphism from $S^{\la 3, 1\ra}$ to $X$, but $S^{\la 3, 1\ra}$ only occurs in a Specht filtration in a Brauer induced module for the Brauer induced module $P_{\Lambda_0}^{\la 3, 1\ra}\cong \Omega^{p-2}(Y^{\la 2\ra})$. On the other hand, $\Omega^{p-2}(Y^{\la 2\ra})$ does not occur in the iterated extension of some module $W_r$ (see Proposition \ref{prop3dot4dot3}). Thus, $\Hom_\Lambda(W_r, X)=0$ for every $r=1, \ldots, 2p-4$. So, (a) holds.
	
By the first part and by applying $\Hom_{\Lambda}(-, X)$ to the exact sequences of Proposition \ref{prop3dot4dot3} we get the following isomorphism 
	$$\Hom_\Lambda(Z_r, X)\cong \begin{cases}
		\Hom_\Lambda(S^{\la a+2, a\ra}, X), &\quad \text{ if } r=2a-1\\
		\Hom_\Lambda(S^{\la a+2, a+1\ra}, X)=0, &\quad \text{ if } r=2a
	\end{cases}. $$
Thus, the Claim (b) follows.
\end{proof}


\begin{Lemma} Assume that $p>3$. We have ${\rm Ext}^i_{\La}(Y^{\la p\ra},  X)=0$ for $1\leq i\leq t-2$ and $\Ext_\Lambda^{t-1}(Y^{\la p\ra},  X)\neq 0$, where $t=2(p-1)$. \label{lemma4dot3}
\end{Lemma}
\begin{proof}
	
	By Theorem \ref{thm2dot2dot1}, $$\Ext_\Lambda^i(Y^{\la p\ra}, X)\cong \Ext_\Lambda^i(S^{\la 2, 1\ra}, Y^{\la p\ra})\cong \Ext^i_{S_k(p, 2p)}(\Cs(\la 2, 1\ra), I(\la p\ra))=0$$ for $i=1, 2$. Now, assume that $i\geq 3$. Recall that the terms $R_r$ in (\ref{eq3}) are in $\add D(Q_\Lambda)$ and $Y^{\la p\ra }$ does not occur for $r\leq t-1$.
	
	By Lemma \ref{lemma4dot2}, applying $\Hom_\Lambda(-, X)$ to (\ref{eq3}) we obtain $\Ext_\Lambda^i(Z_{j-1}, X)\cong \Ext_\Lambda^{i+1}(Z_j, X)$ for all $i\geq 1$ and $2\leq j\leq t$. Hence,  dimension shifting yields $$\Ext_\Lambda^i(Y^{\la p\ra}, X)\cong \Ext_\Lambda^1(Z_{t-i+1}, X).$$
	Since $t-1\geq i\geq 3$, the module $Z_{t-i+1}$ is constructed in Proposition \ref{prop3dot4dot3}. For $t-1\geq i\geq 3$ consider the exact sequence $0\rightarrow Z_{t-i}\rightarrow R_{t-i}\rightarrow Z_{t-i+1}\rightarrow 0$ and apply $\Hom_\Lambda(-, X)$. Hence, by Lemma \ref{lemma4dot2}, we obtain an exact sequence
\begin{align}
	\Hom_\Lambda(R_{t-i}, X)\rightarrow \Hom_\Lambda(Z_{t-i}, X)\rightarrow \Ext_\Lambda^1(Z_{t-i+1}, X)\rightarrow 0. \label{eq16}
\end{align} By Lemma \ref{lemma4dotonedot3}, $\Hom_\Lambda(Z_{t-i}, X)=0$ for $3\leq i\leq t-2$. For $i=t-1$, the exact sequence (\ref{eq16}) becomes
$$\Hom_\Lambda(R_1, X)\rightarrow \Hom_\Lambda(Z_1, X)\cong k\rightarrow \Ext_\Lambda^1(Z_2, X)\rightarrow 0.$$ By Example \ref{example3dot4dot1}, $R_1=P^{\la 4, 2\ra}\bigoplus P^{\la 5, 3\ra}$. So, $\Hom_\Lambda(R_1, X)=0$ and thus $\Ext^{t-1}(Y^{\la p\ra}, X)\cong \Ext_\Lambda^1(Z_2, X)\cong k$.
\end{proof}

	\begin{Theorem}\label{maintheoremone}
		Let $k$ be an algebraically closed field with characteristic $p\geq 5$. Then  $$V^{\otimes 2p}\domdim_{S_k(p, 2p)} S_k(p, 2p)= 4(p-1).$$ In particular, $\Ext_{\Lambda_k(p, 2p)^{op}}^l(V^{\otimes 2p}, V^{\otimes 2p})=\Ext_{\Lambda_0^{op}}^l(Q_{\Lambda_0}, Q_{\Lambda_0})=0$ for $1\leq l\leq 4(p-1)-2$. 
	\end{Theorem}
	\begin{proof} Let $T$ be the characteristic tilting module of $S_k(p, 2p)$.  Then $\add T=\add V^{\otimes 2p}\oplus T(2^p)$. 
	By Theorem 3.1 of \cite{CE24}, we have $V^{\otimes 2p}\domdim_{S_k(p, 2p)} S_k(p, 2p)=2V^{\otimes 2p}\codomdim_{S_k(p, 2p)} T$.
		
		By Theorem 8.1.2 of \cite{Cr2}, $V^{\otimes 2p}\codomdim_{S_k(p, 2p)} T\geq 2$ (see also \citep[Remark 8.1.1]{Cr2}). Recall that $S_k(p, 2p)=A_0\oplus A_1$, where $A_0$ is the principal block of $S_k(p, 2p)$. So, by Corollary 3.1.9 of \cite{Cr2}, we obtain $Q\codomdim_{A_0} T_0\geq V^{\otimes  2p}\codomdim_{S_k(p, 2p)} T\geq 2$ and $\add T_0=\add Q\oplus T(2^p)$. Moreover, the maximal direct summand of $V^{\otimes 2p}$ which is an $A_1$-module is a characteristic tilting module. Write $\Lambda:= \Lambda_k(p, 2p)$. 
		Since $D(Q_\Lambda)$ is self-dual we have the following isomorphisms $$\Ext_\Lambda^i(D(Q_\Lambda), X)=\Ext_\Lambda^i(D(Q_\Lambda), {}^\natural S^{\langle 2, 1\rangle })\cong \Ext_\Lambda^i(S^{\langle 2, 1\rangle }, D(Q_\Lambda))\cong \Ext_{\Lambda^{op}}^i(Q_\Lambda, DS^{\langle 2, 1\rangle }).$$
		
		By Lemmas \ref{lemma4dot2},  \ref{lemma4dot3} and \ref{lemma2dot1dot2}, \begin{align*}
			D\Tor_i^\Lambda(Q_\Lambda, \Hom_{A_0}(Q, T(2^p)))\cong \Ext_{\Lambda^{op}}^i(Q_\Lambda, D\Hom_{A_0}(Q, T(2^p)))\cong\Ext_{\Lambda^{op}}^i(Q_\Lambda, DS^{\langle 2, 1\rangle })=0
		\end{align*} for $1\leq i \leq 2(p-1)-2$. Also, $\Ext_{\Lambda}^{2(p-1)-1}(D(Q_\Lambda), X)\neq 0$. By Theorem 3.1.4 of \cite{Cr2} and Lemma \ref{lemma2dot1dot2},  we have $Q\codomdim_{A_0} T(2^p)=2(p-1)$. By Theorem 3.1 of \cite{CE24}, $Q\domdim_{A_0} A_0=4(p-1)$ and thus $Q\domdim_{S_k(p, 2p)} S_k(p, 2p)=4(p-1)$. This means that $V^{\otimes 2p}\domdim_{S_k(p, 2p)} S_k(p, 2p)=4(p-1)$ (see \citep[Corollary 3.1.9]{Cr2}). The last claim follows now by \citep[Theorem 3.1.4(ii)]{Cr2} and Lemma \ref{lemma2dot1dot2}.
	\end{proof}

	\subsection{A tilting module built from Young modules} \label{tiltingmodulebuiltfromYoungmodules}
	Recall that a module $M$ over a finite-dimensional algebra $A$ is \emph{tilting} if $M$ is self-orthogonal, it has finite projective dimension and there exists an exact sequence $0\rightarrow A\rightarrow X_0\rightarrow\cdots\rightarrow X_t\rightarrow 0$ for some $t$ with $X_i\in \add M$.
	
	From Proposition \ref{5dotonedotone} it follows that $\Lambda_0$ is coresolved with a finite number of terms by direct sums of direct summands of $\bigoplus_{\lambda\neq \la p \ra} Y^\lambda$. Moreover, by Remark \ref{pdimYoungmodules}, $\bigoplus_{\lambda\neq \la p \ra} Y^\lambda$ has projective dimension $p-2$. So, to show that $\bigoplus_{\lambda\neq \la p \ra} Y^\lambda$ is a tilting module, it remains to verify that $\Ext_{\Lambda_0}^i(\bigoplus_{\lambda\neq \la p \ra} Y^\lambda, \bigoplus_{\lambda\neq \la p \ra} Y^\lambda)=0$ for all $i>0$. 
	We show that ${\rm Ext}^i_{\La}(Y^{\lambda}, Y^{\mu})=0$ where
	$Y^{\lambda}$ and $Y^{\mu}$ are  non-projective
	Young modules in $\Lambda_0$, except the case $\lambda =\mu = \la p\ra$. 
	We may assume $\lambda \neq \la p\ra$ since we have by duality
	$${\rm Ext}^i_{\La} (Y^{\mu}, Y^{\lambda}) \cong {\rm Ext}^i_{\La}({}^\natural Y^{\lambda}, {}^\natural Y^{\mu})\cong \Ext_{\La}^i(Y^\lambda, Y^\mu).$$
	
	\begin{Lemma}Let $p>3$. Let $\lambda\neq \la p\ra$ and $\mu$ be two partitions of $2p$ in at most $p$ parts with $p$-weight two. Then,
		$\Ext_{\Lambda}^i(Y^\lambda, Y^\mu)=0$ for all $i>0$.\label{lema4dot2dot1}
	\end{Lemma}
	\begin{proof} 
			
				Assume that $i>p-2$. Then $\Ext^i_{\Lambda}(Y^\lambda, Y^\mu)=0$ because the projective dimension of $Y^\lambda$ is $p-2$ (see Remark \ref{pdimYoungmodules}).

		Assume that $0<i\leq p-2$. By Theorem \ref{maintheoremone}, we obtain
		\begin{align*}
			\Ext^i_{\Lambda}(Y^\lambda, Y^\mu)\subset \Ext_\Lambda^i(DQ, DQ)=0.
		\end{align*}
	\end{proof}

	\begin{Theorem}\label{tiltingtheorem}
		$D(Q_{\Lambda_0})/Y^{\la p\ra}=\bigoplus_{\lambda\neq \la p \ra} Y^\lambda$ is a tilting $\Lambda_0$-module. Moreover, the complement of the direct sum of copies of $Y^{\la p\ra}$ in $DV^{\otimes 2p}$ is a tilting module.
	\end{Theorem}
\begin{proof}
	By Lemma \ref{lema4dot2dot1} and Proposition \ref{5dotonedotone}, $D(Q_{\Lambda_0})$ is a tilting $\Lambda_0$-module. Since the direct summands of $V^{\otimes 2p}$ corresponding to a non-principal block of $S_k(p, 2p)$ are characteristic tilting modules in the respective block, the direct summands of $DV^{\otimes 2p}$  corresponding to a non-principal block of $\Lambda$ are full tilting modules. Thus, the complement of the direct sum of copies of $Y^{\la p\ra}$ in $DV^{\otimes 2p}$ is a tilting $\Lambda_k(p, 2p)$-module.
\end{proof}

	\begin{Remark}
		As recalled in Subsection \ref{sec2dot3}, the right Young modules are given by $DY^\lambda$ and so it follows from Theorem \ref{tiltingtheorem} that the complement of the direct sum of copies of $DY^{\la p\ra}$ (the simple Young right module) in $V^{\otimes 2p}$ is a tilting module (see also Theorem \ref{thm3:3:2}).
	\end{Remark}

	\subsection{The principal block of $S_k(3, 6)$ and $\Lambda_k(3, 6)$ in characteristic three}\label{sec4dot3}
	
In this subsection, our aim is to show that Theorem \ref{maintheoremone} and Theorem \ref{tiltingtheorem} are also valid in characteristic three.
	
	The appendix of \cite{zbMATH07389875} is also valid for $p=3$, that is, we get the description of Specht modules and indecomposable projective modules.
	The only difference occurs in \citep[Corollary B.6]{zbMATH07389875}, in the notation for the Specht module $S^{\la 2, 1\ra}$. In the case $p\geq 5$  the composition factors of $S^{\la 2, 1\ra}$
	are $D^{\la 3, 1\ra}$ and $D^{\la 4, 3\ra}$ but when $p=3$ the second composition factor of $S^{\la 2, 1\ra}$ is $D^{\la 3\ra}$. To write down the quiver, we take the right hand side upper corner of the general quiver, replace the vertices $\la 4, a\ra$ just by $\la a\ra$
	for $1\leq a\leq 3$. So, the Gabriel quiver of the principal block of $\Lambda_k(3, 6)$ is 
		
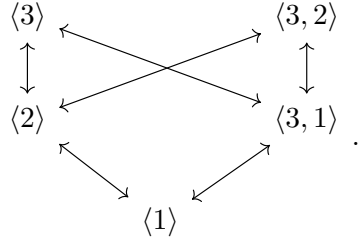
\begin{equation*}
		\begin{tikzcd}
		 \langle 3\rangle \arrow[d, leftrightarrow] \arrow[drr, leftrightarrow]& &  \langle 3, 2\rangle \arrow[d, leftrightarrow]  \\
			 \langle 2\rangle \arrow[dr, leftrightarrow] \arrow[urr, leftrightarrow]& & \langle 3, 1\rangle \\
			  & \langle  1\rangle \arrow[ur, leftrightarrow] \\
		\end{tikzcd}.
	\end{equation*} The quiver and its relations of this algebra was obtained in \citep[Theorem 7.1]{zbMATH00646847} based on \cite{zbMATH04125660}.
Then  for the quotient $\Lambda_k(3, 6)\cong k\cS_6/I_3$ we have, as before, projective resolutions of non-projective and non-simple Young modules ending with 
the projective modules for the quotient which are not injective:
\begin{align}
	0\to P_{\Lambda_0}^{\la 3,2\ra} \to P^{\la 2\ra} \to Y^{\la 1\ra} \to 0 \label{equation18} \\
	0\to P_{\Lambda_0}^{\la 3,1\ra} \to P^{\la 3\ra} \to Y^{\la 2\ra} \to 0 \label{equation19} \\
	0\to P_{\Lambda_0}^{\la 1 \ra} \to P^{\la 2\ra} \to Y^{\la 3,2 \ra} \to 0. \label{equation20}
\end{align}

	To summarise, the projective and Young modules are the following:
	
	\begin{enumerate}[(i)]
		\item $P^{\la 3\ra} = Y^{\la 2, 1\ra}$ and $P^{\la 2\ra} = Y^{\la 3, 1\ra}$ and they are both projective and injective.
		\item $P_{\Lambda_0}^{\la 3, 2\ra} = U{\la 3,2\ra\choose \la 3, 1\ra}$, it is Brauer induced and it is isomorphic to  $\Omega Y^{\la 1\ra}$.
		\item $P_{\Lambda_0}^{\la 3,1\ra}  = U{\la 3,1\ra \choose \la 2, 1\ra}$, it is Brauer induced and it is isomorphic to $\Omega Y^{\la 2\ra}$. 
		\item $P_{\Lambda_0}^{\la 1\ra} = U{\la 1\ra \choose \la 3, 1\ra}$ and this is Brauer induced and is isomorphic to $\Omega Y^{\la 3, 2\ra}$.
		\item $Y^{\la 3 \ra}= S^{\la 3 \ra}\cong k$.
		\item The modules $Y^{\la 1\ra}$ and $Y^{\la 2\ra}$ and $Y^{\la 3, 2\ra}$ are the Young modules which are neither projective nor simple, and with respect to their Specht filtration, they are given as follows:
		$$Y^{\la 1 \ra}\cong \begin{matrix}  S^{\la 2\ra}  \cr 
			S^{\la 1\ra}  \end{matrix}, \quad Y^{\la 2\ra}\cong  \begin{matrix}  S^{\la 3\ra}  \cr 
			S^{\la 2\ra}  \end{matrix}, \quad Y^{\la 3, 2\ra}\cong  \begin{matrix}  S^{\la 2\ra}  \cr 
			S^{\la 3, 2\ra}  \end{matrix}. $$
	\end{enumerate}

	In particular, we have exactly the same exact sequences for $p=3$ as described in Proposition \ref{5dotonedotone}. It follows that the non-projective Young modules $Y^\lambda$ with $\lambda\neq \la 3\ra$ have projective dimension equal to one. We also have that $S^{\la 2, 1\ra}$  has length two, with top isomorphic to $D^{\la 3, 1\ra}$ and socle isomorphic to $D^{\la 3\ra}$.

	\begin{Lemma} Assume $p=3$ and let $X$ be the dual Specht module ${}^\natural S^{\la 2, 1 \ra}$. Then 
		\begin{enumerate}[(a)]
			\item If $\lambda \neq \la 3\ra$, then ${\rm Ext}^i_\Lambda(Y^{\lambda}, X)=0$ for all $i\geq 1$.
			\item We have ${\rm Ext}_\Lambda^i(Y^{\la 3\ra}, X)=0$ for $1\leq i \leq  2$ and ${\rm Ext}_\Lambda^3(Y^{\la 3\ra}, X)\neq 0$. 
		\end{enumerate}\label{4dot3dot1}
	\end{Lemma}

	\begin{proof} \  We can assume that   $Y^{\lambda}$ is not projective.  Recall that $X$ has length two, with top isomorphic to $D^{\la 3\ra}$ and socle
		$D^{\la 3, 1\ra}$. To prove (a), assume that $\lambda \neq \la 3\ra$. By Equations (\ref{equation18}), (\ref{equation19}) and (\ref{equation20}),
		 $Y^{\lambda}$ has projective dimension $1=p-2$ and we only need the case $i=1$. 
		We apply the functor ${\rm Hom}_{\La}(-, X)$ to the minimal projective resolution of
		$Y^{\lambda}$.  If $\lambda = \la 1\ra$ or $\la 3,2\ra$ then this gives the sequence which is identically zero and hence ${\rm Ext}^1_\Lambda(Y^{\lambda}, X)$ is zero.  This leaves
		$$0\to {\rm Hom}_\Lambda(Y^{\la 2\ra}, X)\to {\rm Hom}_\Lambda(P^{\la 3\ra}, X) \to {\rm Hom}_\Lambda(P^{\la 3,1\ra}_{\Lambda_0}, X) \to  {\rm Ext}_\Lambda^1(Y^{\la 2 \ra}, X) \to 0.$$
		The second and third terms are isomorphic to $k$ and the first term is zero, hence the Ext space is zero. So (a) holds.
		
		For (b), we can now use part (a).
 We apply ${\rm Hom}_\Lambda(-, X)$ to the exact sequence 
		$$0\to S^{\la 2\ra}\to Y^{\la 2\ra} \to Y^{\la 3\ra} \to 0.$$
		In the resulting long exact sequence, all Hom spaces are zero and hence ${\rm Ext}_\Lambda^1(Y^{\la 3\ra}, X)=0$. Moreover,
		\begin{align}
			{\rm Ext}^i_\La(Y^{\la 3\ra}, X)\cong {\rm Ext}^{i-1}_\La(S^{\la 2\ra}, X) \label{e17}
		\end{align}
		for $i=2, 3$. 
		Now we apply ${\rm Hom}_\La(-, X)$ to the exact sequence
		$$0\to S^{\la 1\ra}\to Y^{\la 1\ra} \to S^{\la 2\ra}\to 0.$$
		In the resulting long exact sequence, the Hom spaces are zero and thus ${\rm Ext}_\Lambda^1(S^{\la 2\ra}, X)=0$. Moreover,
		\begin{align}
			{\rm Ext}_\La^2(S^{\la 2\ra}, X) \cong {\rm Ext}_\La^1(S^{\la 1\ra}, X). \label{eq18}
		\end{align}
		Applying ${\rm Hom}_\La(-, X)$ to the exact sequence
		$0\to S^{\la 3, 1\ra} \to P_{\La}^{\la 1\ra} \to S^{\la 1\ra}\to 0$
		gives
		$$0 = {\rm Hom}_\La(P_{\La}^{\la 1\ra}, X ) \to {\rm Hom}_\La(S^{\la 3, 1\ra}, X) = k \to {\rm Ext}_\La^1(S^{\la 1\ra}, X) \to 0$$
		and ${\rm Ext}^1_\La(S^{\la 1\ra}, X)=k$. By (\ref{e17}) and (\ref{eq18}), we obtain that $\Ext_\Lambda^2(Y^{\la 3\ra}, X)\cong \Ext_\La^1(S^{\la 2\ra}, X)=0$ and $\Ext_\Lambda^3(Y^{\la 3\ra}, X)\cong \Ext_\La^1(S^{\la 1\ra}, X)\cong k$.
	\end{proof}
	
	\begin{Cor}\label{4dot3dot2}
			Let $k$ be an algebraically closed field with characteristic $3$. Write $\Lambda=\End_{S_k(3, 6)}(V^{\otimes 6})^{op}$. Then, the following assertions hold.
			\begin{enumerate}[(a)]
				\item $V^{\otimes 6}\domdim_{S_k(3, 6)} S_k(3, 6)= 8$. In particular, $\Ext_{\Lambda^{op}}^l(V^{\otimes 6}, V^{\otimes 6})=\Ext_{\Lambda_0^{op}}^l(Q, Q)=0$ for $1\leq l\leq 6$. 
				\item Let $\lambda\neq \la p\ra$ and $\mu$ be two partitions of $6$ in at most $3$ parts with $3$-weight two. Then,
				$\Ext_{\Lambda}^i(Y^\lambda, Y^\mu)=0$ for all $i>0$.
				\item $DQ/Y^{\la 3\ra}=\bigoplus_{\lambda\neq \la 3 \ra} Y^\lambda$ is a tilting $\Lambda_0$-module. Moreover,
				the complement of the direct sum of copies of $Y^{\la 3 \ra}$ in $DV^{\otimes 6}$ is a tilting module.
			\end{enumerate} 
	\end{Cor}
\begin{proof}
	The exact same arguments employed in Theorem \ref{maintheoremone}, Theorem \ref{tiltingtheorem} and Lemma \ref{lema4dot2dot1} hold now for characteristic three by replacing the use of Lemma \ref{lemma4dot2} by Lemma \ref{4dot3dot1}(a) and Lemma \ref{lemma4dot3} by Lemma \ref{4dot3dot1}(b).
\end{proof}

\subsection{The relative dominant dimension over arbitrary fields}

As illustrated in \cite{Cr2}, relative dominant dimension is a homological invariant that is preserved under base change to an algebraically closed field. Recall that both the Schur algebra and the tensor power admit base change properties:
$\overline{k}\otimes_k (k^n)^{\otimes d}\cong \left( \overline{k}^n\right) ^{\otimes d}$ and $\overline{k}\otimes_k S_k(n, d)\cong S_{\overline{k}}(n, d)$
for every field $k$, where $\overline{k}$ denotes the algebraic closure of $k$ (see for example \citep[2.5]{Gr}.)

Thus, the results of this section on relative dominant dimension can be summarised as follows:

\begin{Theorem} \label{maintheoremarbitraryfield}
	Let $k$ be an arbitrary field with positive characteristic $p$.  Then $$V^{\otimes 2p}\domdim_{S_k(p, 2p)} S_k(p, 2p)=4(p-1).$$
\end{Theorem}
\begin{proof}
	For $p=2$, this is contained in \citep[Theorem 5.8]{CE24}. For $p=3$, the result follows by  \citep[Lemma 3.2.3]{Cr2} and Corollary \ref{4dot3dot2}. For $p>3$, the result follows by \citep[Lemma 3.2.3]{Cr2} and Theorem \ref{maintheoremone}.
\end{proof}

		\section{Homological properties of $S_k(p, 2p)$ and $\Lambda_k(p, 2p)$}\label{homologicalpropertiesofschuralgebras}
	
	In this section, our goal is to better understand the homological structure of Schur algebras of the form $S_k(p, 2p)$. More precisely, we will determine (over fields of characteristic $p$)
	\begin{itemize}
		\item the Hemmer-Nakano dimension of $\mathcal{F}(\St_{R_k(p, 2p)})$ of the quasi-hereditary cover of $\Lambda_k(p, 2p)$, the centraliser of ${}_{S_k(p, 2p)} V^{\otimes 2p}$, which is formed by the Ringel dual of $S_k(p, 2p)$;
		\item a quasi-precluster tilting module of $\Lambda_k(p, 2p)$ in the sense of \citep{CP2}.
	\end{itemize}
	The latter is done by showing that the Schur algebra $S_k(p, 2p)$ fits into a non-trivial relative Auslander pair. In particular, we find a new way to determine the global dimension of $S_k(p, 2p)$.

	\subsection{The quasi-hereditary cover of $\Lambda_k(p, 2p)$} Recall the notation introduced in Subsection \ref{sec2dot3}.
	As constructed in \cite{Cr2} and recalled in Theorem \ref{thm2dot2dot1} the pair $$(R_k(p, 2p), \Hom_{S_k(p, 2p)}(T, V^{\otimes 2p}))$$ is a quasi-hereditary cover of $\Lambda_k(p, 2p)$. Associated with this cover comes the Schur functor $$F_{p, 2p}=\Hom_{R_k(p, 2p)}(\Hom_{S_k(p, 2p)}(T, V^{\otimes 2p}), -)\colon R_k(p, 2p)\m\rightarrow \Lambda_k(p, 2p)\m.$$
	
	 But, up until now, only a lower bound for the  Hemmer-Nakano dimension of $\mathcal{F}(\St_{R_k(p, 2p)})$  was known (see Theorem \ref{thm2dot2dot1}). 
	We apply now our computation of the relative dominant dimension $V^{\otimes 2p}\domdim S_k(p, 2p)$ to this cover.

	\begin{Theorem}\label{quasihereditarycover}
	Let $k$ be a field of characteristic $p>0$. Let $R_k(p, 2p)$ be the Ringel dual of $S_k(p, 2p)$. 
	Then $(R_k(p, 2p), \Hom_{S_k(p, 2p)}(T, V^{\otimes 2p}))$ is a $2(p-2)$-faithful quasi-hereditary cover of $\Lambda_k(p, 2p)$ in the sense of \cite{Rou}.
	That is, the Schur functor $F_{p, 2p}$ induces isomorphisms $$\Ext_{R_k(p, 2p)}^i(M, N)\cong \Ext_{ \Lambda_k(p, 2p)}^i(F_{p, 2p}M, F_{p, 2p}N)$$ for every $M, N\in \mathcal{F}(\St_{R_k(p, 2p)})$ and $0\leq i \leq 2(p-2)$; the isomorphism fails in general for $i=2(p-2)+1.$
\end{Theorem}
\begin{proof}
	For $p=2$, this is contained in  \citep[Corollary 6.8]{CE24}. Assume now that $p>2$. Then, the result follows from Theorem \ref{maintheoremarbitraryfield}, \citep[Theorem 5.3.1]{Cr2} and \citep[Theorem 3.1]{CE24}.
\end{proof}

This means that the Hemmer-Nakano dimension of $\mathcal{F}(\St_{R_k(p, 2p)})$ (with respect to $F_{p, 2p}$) is exactly $2(p-2)$.

As an application of this fact, we obtain the following.

	\begin{Cor}
	The Young  $\Lambda_0$-module $Y^{\la p \ra}$ has infinite projective dimension. \label{simpleYoungmoduledimension}
\end{Cor}
\begin{proof}
	Consider $\lambda$ so that $Y^\lambda$ has finite projective dimension. By Corollary \ref{finitisticlamdbazero}, \linebreak\mbox{$\pdim_{\Lambda_0} Y^\lambda\leq 2p-4$.} In particular, the projective dimension of $Y^\lambda$ is the maximum natural number $n$ satisfying $\Ext^n_{\Lambda_0}(Y^\lambda, \Lambda_0)\neq0$. By Theorem \ref{quasihereditarycover}, $\Ext^i_{\Lambda_0}(Y^\lambda, \Lambda_0)\cong \Ext^i_{A_0}(I(\lambda), Q)$ for every $i=0, \ldots, 2p-4$ and this isomorphism is induced by the functor $\Hom_{A_0}(Q, -)$. Moreover, by projectivisation, $\Hom_{A_0}(Q, -)$ identifies the modules in $\add Q$ with the projective  $\Lambda_0$-modules.  Thus, if $\pdim_{\Lambda_0} Y^\lambda=i\leq 2p-4$ then there exists an $\add Q_\Lambda$-resolution of $I(\lambda)$ of length $i$. Since $Q$ is self-orthogonal, this means that $Q\codomdim_{A_0} I(\lambda)=+\infty$.  By the proof of Proposition \ref{5dotonedotone}, all Young modules $Y^\lambda$ with $\lambda\neq \la p\ra$ have finite projective dimension. So, if $Y^{\la p\ra}$ has also finite projective dimension, then we would have $Q\codomdim_{A_0} I(\lambda)=+\infty$ for every partition $\lambda$ in $\Lambda_0$. But then, $Q\domdim_{A_0} A_0=Q\codomdim_{A_0} DA_0=+\infty$ which would contradict Theorem \ref{maintheoremone}. Thus, $Y^{\la p\ra}$ has infinite projective dimension.
\end{proof}

Another application of Theorem \ref{quasihereditarycover} is the cellularity of the endomorphism algebra of the tilting module built in the previous section.

	\begin{Prop}Assume that $p>2$.
	The algebra $\displaystyle E_p=\End_{\Lambda_k(p, 2p)}(\bigoplus_{\lambda\neq \langle p\rangle } Y^{\lambda})$ is cellular with infinite global dimension.
\end{Prop}
\begin{proof}
	By Theorem \ref{quasihereditarycover}, $$\End_{\Lambda_k(p, 2p)}(\oplus_{\lambda\neq \langle p\rangle } Y^{\lambda})\cong \End_{\Lambda_k(p, 2p)}(\oplus_{\lambda\neq \langle p\rangle } F_{p, 2p}T_R(\lambda))\cong \End_{R_k(p, 2p)}(\oplus_{\lambda\neq \langle p\rangle } T_R(\lambda)).$$
	Thus, by \citep[Theorem 7.5]{appendix} (see also \cite{zbMATH06790119} and \cite{zbMATH07537690}), $E_p$ is a cellular algebra. Since $\oplus_{\lambda\neq \langle p\rangle } Y^{\lambda}$ is a tilting module, the algebras $E_p$ and $\Lambda_k(p, 2p)$ are derived equivalent. Since $\Lambda_k(p, 2p)$ has infinite global dimension, it follows that $E_p$ has infinite global dimension.
\end{proof}

	\subsection{A quasi-precluster tilting module over $\Lambda_0$} \label{quasipreclustertiltingmodule}
	
	In this subsection, we present the underlying reason and the phenomenon that explains why the relative dominant dimension \linebreak $V^{\otimes 2p}\domdim_{S_k(p, 2p)} S_k(p, 2p)$ is precisely $4(p-1)$ and what does this mean for both the structure of $\Lambda_0$ and the principal block of $S_k(p, 2p)$.
	 As we will show in the coming results, these algebras fit in the setup presented in \cite{CP1} and \cite{CP2}.
	
	To see that it fits in the setup presented in \cite{CP1} it is enough to combine the work of the previous section with  \cite{zbMATH02005565} 
	
	\begin{Theorem}
		Let $k$ be an algebraically closed field of characteristic $p\geq 3$. Then,
 $(S_k(p, 2p), (k^p)^{\otimes 2p})$ is a relative $4(p-1)$-Auslander pair. 
	\end{Theorem}
\begin{proof}
	By Theorem \ref{maintheoremarbitraryfield} and \citep[Theorem 5.9]{zbMATH02005565}, $$V^{\otimes 2p}\domdim_{S_k(p, 2p)} S_k(p, 2p)=4(p-1)=\gldim S_k(p, 2p).$$
\end{proof}

Analogously, this argument can be applied to the principal block of $S_k(p, 2p)$. However, the previous result, by itself, does not give everything that we want to know about the relative Auslander pair, in particular it does not give the value of the projective dimension of $V^{\otimes 2p}$. To address this, we make use of the resolutions built in Subsection \ref{coresolvingsingularSpechtmodule} to give a new approach to compute the global dimension of the principal block of $S_k(p, 2p)$ and mainly to show that these algebras fit in the setup presented in \cite{CP2}.
	
		\begin{Theorem} \label{maintheorempair}Let $k$ be an algebraically closed field of characteristic $p\geq 3$.
		Let $A_0$ be the principal block of $S_k(p, 2p)$ and $Q$ the direct summand of $V^{\otimes 2p}$ corresponding to the principal block. Then $(A_0, Q)$ is a relative $4(p-1)$-Auslander pair and $\pdim_{A_0} Q=p-2$. In particular, $\gldim A_0=4(p-1)$.
	\end{Theorem}
\begin{proof} $Q$ is a direct summand of a characteristic tilting module, therefore it is a self-orthogonal $A_0$-module. 
	For the first claim, we need to show that $\gldim A_0\leq 4(p-1)\leq Q\domdim_{A_0} A_0$. The second inequality holds by Theorem \ref{maintheoremone} and Corollary \ref{4dot3dot2}. So, we determine now $\gldim A_0$. To do this, it is actually enough to show that $\gldim A_0\leq 4(p-1)$. Indeed,  since $Q\domdim_{A_0} A_0$ is a finite number (in particular $Q$ is not a characteristic tilting module) we have $\gldim A_0\geq Q\domdim_{A_0} A_0=4(p-1)$.

Assume that $p=3$. Then the global dimension of $S_k(3, 6)$ is $4(3-1)=8$ (see \cite{P}). Since $A_0$ is a block component of $S_k(3, 6)$, then  $\gldim A_0\leq \gldim S_3(3, 6)=8$.

Assume now that $p\geq 5$. Our strategy in this case is to use the quasi-hereditary structure of $S_k(p, 2p)$ and determine an upper bound for the injective dimension of the characteristic tilting module of $A_0$. Fix $\Lambda:=\Lambda_k(p, 2p)$.
Let $T_0$ be the characteristic tilting module of $A_0$, that is, $T_0=Q\oplus T(2^p)$.  As discussed in the proof of Theorem \ref{maintheoremone}, $Q\codomdim_{A_0} T_0\geq 2$. Lemma 2.8(i) of \cite{CP2} shows that $\injdim_{A_0} T_0=\dim_{\add Q_\Lambda} D\Hom_{A_0}(Q, T_0)$ (see also Subsection \ref{Homological dimensions}). By Lemma \ref{4dot1dot1}, $$D\Hom_{A_0}(Q, T_0)\cong D\Hom_{A_0}(Q, Q)\oplus D\Hom_{A_0}(Q, T(2^p))\cong D\Lambda_0\oplus DS^{\la 2, 1 \ra}$$ as right $\Lambda_0$-modules.

Since $\Ext_\Lambda^i(D(Q_\La), D(Q_\La))=0=\Ext_\Lambda^i(Q_\La, Q_\La)=0$ for $1\leq i\leq 4(p-1)-2$ and $R_{2(p-1)}=Y^{\langle p\rangle}\in \add DQ$ we obtain that (\ref{eq3}) is exact under $\Hom_\Lambda(-, D(Q_\La))$. Write $\delta$ to denote the exact sequence (\ref{eq3}). Then $D\delta$ is exact and it remains exact under $\Hom_\Lambda(Q_\La, -)$ and this means that $\dim_{\add Q_\Lambda} DS^{\la 2, 1\ra}\leq 2(p-1)$. Denote by $\gamma$ the direct sum of all the exact sequences (1)-(5) appearing in the proof of Proposition \ref{5dotonedotone}. This is an $\add D(Q_\La)$-coresolution of the direct sum of the projective $\Lambda_0$-modules which are not Young modules. Since the size of this coresolution is $p-2$ and $\Ext_{\Lambda_0}^i(Q, Q)=0$ for $1\leq i\leq p\leq  4(p-1)-2$, $\gamma$ remains exact under $\Hom_{\Lambda_0}(-, DQ)$. Thus, $D\gamma$ is exact and it remains exact under $\Hom_{\Lambda_0}(Q, -)$. So, $\dim_{\add Q_\Lambda} D\Lambda_0\leq p-2\leq 2(p-1)$ and it follows that $\injdim_{A_0} T_0\leq 2(p-1)$. The simple preserving duality on $A_0$ imposes then that $\gldim A_0\leq 4(p-1)$ (see \cite{zbMATH02105773}).

	Observe that $\dim_{\add Q} D\Lambda$ cannot be smaller than $p-2$ since the coresolutions constructed in the proof of Proposition \ref{5dotonedotone} are sent under $\Hom_\Lambda(-, D(Q_\La))$ to projective resolutions of direct summands of $D(Q_\La)$ whose middle terms are indecomposable modules. Hence, such projective resolutions are minimal and thus $\pdim_{A_0} Q=p-2$.

	For $p=3$, the resolutions in Subsection \ref{sec4dot3} have length one, and since $Q_\Lambda$ is self-orthogonal up to degree $4(p-1)-2$, those resolutions remain exact under $\Hom_\Lambda(-, DQ)$ as well. So, $\dim_{\add Q} D\Lambda=1$ in characteristic three and so $\pdim_{A_0} Q=1=3-2$ using the equality established in Lemma 2.8(i) of \cite{CP2}.
\end{proof}

	So according to \citep[Theorem 3.3]{MTcorrespondence}, Proposition \ref{5dotonedotone} can be reformulated into saying that $Q_\La$ is an $(p-2)$-quasi-cogenerator over $\Lambda$ and its self-duality implies that it is also an $(p-2)$-quasi-generator over $\Lambda$.

From the above theorem it follows that $$\gldim A_0=4(p-1)\geq 2(p-1)=2p-2=(p-2)+(p-2)+2=\pdim_{A_0} Q+\injdim_{A_0} Q+2.$$ This means that these pairs fit into the higher dimensional Auslander-Iyama-Solberg correspondence established in \cite{CP2}. Further, $Q_\La$ is a $(4(p-1), p-2, p-2)$-quasi-precluster tilting $\Lambda$-module in the sense of \cite{CP2}. In particular, the module $\tau\Omega^{3p-4}Q_\La\cong \tau\Omega^{3p-4}Y^{\la p\ra}$ fits into an $\add Q_\La$-resolution of length $p-2$ and 
$$\add Q_\Lambda={}^{\perp_{p-2}} Q_\Lambda\cap Q_\Lambda^{\perp_{3p-4}}={}^{\perp_{3p-4}} Q_\Lambda\cap Q_\Lambda^{\perp_{p-2}}.$$

	\subsection{Global dimension of $S_k(p, 2p)$ and the projective dimension of $V^{\otimes 2p}$} \label{globaldimensionofSp2p} For $p\in \{2, 3\}$ the global dimension of $S_k(p, 2p)$ was determined in \cite{P}. For $p>3$, this homological invariant was determined in \citep{zbMATH02005565} using the theory of algebraic groups. However, in such cases the projective dimension (and injective dimension) of the tensor space remains to be determined. In this section, we illustrate another way to compute $\gldim S_k(p, 2p)$ based on the cellular structure of the symmetric group, and as by-product we obtain the projective and injective dimension of the tensor space. 
		With our approach we have shown already that the global dimension of the principal block of $S_k(p, 2p)$ is $4(p-1)$, so we just need to check the global dimension of the non-principal blocks of $S_k(p, 2p)$. 
		
		\subsubsection{Blocks of $S_k(n, d)$} 
	
	Since $(k^n)^{\otimes d}$ affords a double centraliser property between $S_k(n, d)$ and $\Lambda_k(n, d)$ the blocks of $S_k(n, d)$ are in one to one correspondence with the blocks of $\Lambda_k(n, d)$ in the sense that if $\mathcal{B}$ is a block of $\Lambda_k(n, d)$, then $\Hom_{S_k(n, d)}((k^n)^{\otimes d}, -)$ sends the injective modules of the respective block ${S_k(n, d)}_\mathcal{B}$ to the Young modules living in the block $\mathcal{B}$ (for example, this follows a similar argument as in \citep[Corollary 5.38]{zbMATH01361790} together with Equation (\ref{eqone})). Moreover, if $X$ is the maximal multiplicity-free direct summand of $(k^n)^{\otimes d}$ living in the block $\mathcal{B}$, then the block component of the Schur algebra corresponding to the block $\cB$ is Morita equivalent to $\End_{\mathcal{B}}(X)$. Further, $\mathcal{B}$ can be written as a quotient of a block component of $k\cS_d$, say $\mathcal{K}$, and $\End_{\mathcal{B}}(X)\cong \End_{\mathcal{K}}(X)$. So the block component ${S_k(n, d)}_\mathcal{B}$ is also completely determined by the block $\mathcal{K}$.
	
	We recall the description of the blocks of finite representation-type of $S_k(n, d)$ described in \cite{E2}.
	
		\begin{Lemma} Let $k$ be an algebraically closed field of characteristic $p$, and let $\cB$ be a block of finite representation-type in the group algebra $k\cS_d$. Assume that $n \leq d$.  \label{lemma5dot3dot1}
			\begin{enumerate}[(a)]
				\item The  partitions in $\cB$ are linearly ordered by the dominance order, say $\mu_1>\mu_2>\ldots > \mu_{p}$.
				\item Assume  that $m$ is the largest natural number such that $\mu_m$ has  at most $n$ parts. Then the partitions
				$\mu_1, \ldots, \mu_{m}$  have at most $d$ parts and $\mu_k$ has more than $n $ parts for $k>m$.
				\item Let $S_k(n,d)_{\cB}$ be the component of the Schur algebra corresponding to the block $\cB$. Then 
				$S_k(n,d)_{\cB}$ is Morita equivalent to the algebra $\cA_m$ as defined in \cite{E2}.  In particular, it has finite representation-type and $m$ is smaller than or equal to the number of simple modules in $\cB$.
			\end{enumerate}
	\end{Lemma}
\begin{proof}
	For (a) and (b) see \citep[4.1]{E2}.
	
	The Young modules in $\cB$ are as follows: 
	We have $Y^{\mu_1} = S^{\mu_1}$ which is simple, and  for \mbox{$2\leq i\leq p-1$,} the Young module $Y^{\mu_i}$ is the indecomposable projective module 
	 with Specht subquotients
	$S^{\mu_{i-1}}$ and $S^{\mu_i}$. 
The block component $S(n,d)_{\cB}$ is Morita equivalent to the endomorphism algebra of $\oplus_{i=1}^m Y^{\mu_i}$. 
	Then part (c) follows directly from \citep[Proposition 3.2]{E2}, noting that the Young modules have the appropriate submodule structure.
\end{proof}

	\subsubsection{Non-principal blocks of $S(p, 2p)$}
	To understand the form of the non-principal blocks of $S_k(p, 2p)$, we will analyse the size of the blocks of $\Lambda_k(p, 2p)$. In turn, they arise as quotients of blocks of $k\cS_{2p}$. By \emph{Nakayama's conjecture on the blocks of the symmetric group} and \cite{zbMATH00023610}, the blocks of $k\cS_{2p}$ are determined up to Morita equivalence by their $p$-core and $p$-weight (see also \cite{zbMATH03519005}). The Specht modules in $\Lambda$ are labelled by partitions of $2p$ in at most $p$ parts. The (combinatorial) weight of a block is the number of $p$-hooks removed to reach the $p$-core. So, for partitions $\lambda\in \L^+(p, 2p)$ the weight is either one or two or zero. If the weight is zero, then the block is semisimple, and thus it can be ignored. If the weight is two, then the $p$-core is empty and in such a case we would get the principal block $\Lambda_0$. So, every non-principal (non-semisimple) block component of $\Lambda$ is the quotient of a block of $k\cS_{2p}$, say $\mathcal{K}$, with $p$-weight  equal to one. Since $p$ is larger than the weight of $\mathcal{K}$, the defect of the block $\mathcal{K}$ is equal to the $p$-weight, hence in this case it is one (see for instance \cite{MR644144}). 
	By \citep[Example 1]{zbMATH00023610}, the block $\mathcal{K}$ is Morita equivalent to the principal block of $k\cS_p$, and in particular it has finite representation-type. So, the number of simple modules in $\mathcal{K}$ is $p-1$ (see for instance \cite{E2}).
	By Lemma \ref{lemma5dot3dot1}, the block component of the Schur algebra $S_k(p, 2p)$ corresponding to the block $\mathcal{K}$ is Morita equivalent to $\mathcal{A}_m$ for some $m\leq p-1$.

By direct computation, making use of the quiver and relations of $\mathcal{A}_m$, it is not difficult to see that the global dimension of $\mathcal{A}_m$ is precisely $2(m-1)$, and that the dominant dimension of $\mathcal{A}_m$ is also $2(m-1)$ (compare with \citep[Theorem 1.1]{zbMATH07463798} when the $p$-weight is one). Hence, it is actually a higher Auslander algebra of global dimension $2(m-1)$. Thus,  any block component of $S_k(p, 2p)$ distinct from the principal block has global dimension smaller than or equal to $2(p-2)\leq 4(p-1).$

	Combining this discussion with Theorem \ref{maintheorempair}, we showed the following.
	
	\begin{Theorem}\label{thm5dotone}
		The global dimension of $S_k(p, 2p)$ is equal to $4(p-1)$.
	\end{Theorem}
\begin{proof}
	The global dimension of $S_k(p, 2p)$ is the supremum of the global dimensions of all its block components. So by  Theorem \ref{maintheorempair} and the discussion above, the claim follows.
\end{proof}
	In the example below, we illustrate that the number of simple modules in a block of $\Lambda$ can be smaller than $p-1$, and thus the global dimension of the respective block component of $S_k(p, 2p)$ can be strictly smaller than $2(p-2)$.
	
		\begin{Example} Let $p=5$, then there are two $5$-cores of size $5$ (with $5$-weight one), they are $(3, 2)$ and $(2^2, 1)$.  We can use an abacus to display the partitions in the block $B_{(3,2)}$ that contains $(3, 2)$.
		
		Take an abacus $\Gamma = \la 1,1,2,1,2\ra$ (ie it has one bead on runners $1,2,4$ and two beads on runners $3$ and $5$. 
		Then we label the partitions exactly as we have done before. Then they are in dominance order by
		$$\la 5\ra > \la 3\ra > \la 4\ra > \la 2\ra > \la 1\ra.$$
		In terms of partitions, 
		$$(8, 2) > (6, 4) > (3^3, 1)  > (3, 2^2, 1^3) > (3, 2, 1^5).
		$$
		The last two have more than $5$ parts. The quotient has then three simple modules.
	\end{Example}

\subsubsection{The projective dimension of the tensor power}

\begin{Cor}
	The pair $(S_k(p, 2p), V^{\otimes 2p})$ is a relative $4(p-1)$-Auslander pair and $$\pdim_{S_k(p, 2p)} V^{\otimes 2p}=\injdim_{S_k(p, 2p)} V^{\otimes 2p}=\begin{cases}
		p-2, &\text{ if } p>2\\
		1, &\text{ if } p=2
	\end{cases}.$$
\end{Cor}
\begin{proof} For $p=2$ we refer the reader to \cite{CE24} and \citep[Example 4.1]{CP2}. Assume now that $p>2$.
	Since $\la 2, 1\ra=(2^p)$ is the only non $p$-regular partition of $2p$ in at most $p$ parts, the direct summands of $V^{\otimes 2p}$ in the non-principal block components are characteristic tilting modules. Let $\mathcal{B}$ be a non-principal block of $S_k(p, 2p)$ and $X$ the maximal multiplicity-free direct summand of $V^{\otimes 2p}$ in the block component $\mathcal{B}$. Since $X$ is a full tilting module, $X \domdim \mathcal{B}=+\infty$. By Theorem \ref{maintheoremone}, it follows that $V^{\otimes 2p}\domdim S_k(p, 2p)=4(p-1)$ (see also \citep[Corollary 3.1.9]{Cr2}). By Theorem \ref{thm5dotone}, the first claim follows. 
	
	By Theorem \ref{maintheorempair}, $\pdim_{A_0} Q=p-2$. By the discussion above, $\gldim \mathcal{B}\leq 2(p-2)$. Since $X$ is a characteristic tilting module, we obtain by \citep[Corollary 1]{zbMATH02105773} that $\pdim X\leq p-2$. It follows that $\pdim_{S_k(p, 2p)} V^{\otimes 2p}=p-2$. Since it is self-dual, the same statement holds for the injective dimension.
\end{proof}

\subsubsection{Global dimension of Schur algebras $S(t, 2p)$}

In \cite{zbMATH00966941}, Totaro proved that $S_k(2p, 2p)$ has global dimension equal to $2(2p-2)=4(p-1)$ when $k$ has characteristic $p>2$ and $\gldim S_k(4, 4)=2(4-1)=6$ when $k$ has characteristic two. So far, our approach unravelled the value of $\gldim S_k(p, 2p)$ among other insights into the homological properties of $S_k(p, 2p)$ like the projective dimension of the tensor space. Now, we provide an alternative approach  to assert that $\gldim S_k(p, 2p)=4(p-1)=\gldim S_k(2p, 2p)$ when $p>2$ combining Totaro's result with our results on relative dominant dimensions making use of Schur functors and the simple preserving duality. As a by-product, we obtain the value of all intermediate Schur algebras between $S_k(p, 2p)$ and $S_k(2p, 2p)$.

\begin{Cor}\label{cor5dot3dot5}
	Let $k$ be an algebraically closed field with characteristic $p>0$. Then, $\gldim S_k(t, 2p)=4(p-1)$ for every $t=p, \ldots, 2p-1$.
\end{Cor}
\begin{proof}
	By \cite{Gr}, there exists an idempotent $e$ of $S_k(t, 2p)$ such that $S_k(p, 2p)\cong eS_k(t, 2p)e$.  This idempotent satisfies 1.5 and 1.6 of \cite{E1}. Let $T_t$ be the multiplicity-free characteristic tilting module of $S_k(t, 2p)$ and $\Cs_t$ be the direct sum of all indecomposable costandard modules of $S_k(t, 2p)$ for $t=p, \ldots, 2p-1$  By 1.6 of \cite{E1}, by multiplying an $\add T_t$ resolution of costandard modules with the idempotent $e$, the resolution remains exact and the middle terms are sent to $\add T_p$ (potentially some of the middle terms of the resolution become zero). This means that $\dim_{\add T_p} \Cs_p\leq \dim_{\add T_t} \Cs_t$.
	By Lemma 2 of \cite{zbMATH02105773}, we obtain
	$$\pdim_{S_k(p, 2p)} T_p=\dim_{\add T_p} \Cs_p\leq \dim_{\add T_t} \Cs_t=\pdim_{S_k(t, 2p)}=\pdim_{S_k(t, 2p)} T_t.$$
	
	By the main result of \cite{zbMATH02105773}, we then obtain that $\gldim S_k(p, 2p)\leq \gldim S_k(t, 2p)$. With this, Theorem \ref{thm5dotone} gives $\gldim S_k(t, 2p)\geq 4(p-1)$ \footnote{this bound can also be deduced from  Theorem \ref{maintheoremone} (and from \citep[Theorem B]{CE24} for characteristic two) since in this case the inequality $\gldim S_k(p, 2p)\geq V^{\otimes 2p}\domdim_{S_k(p, 2p)} S_k(p, 2p)$ holds.}. By Theorem 3 of \cite{zbMATH00966941},
	we obtain that $$ 4(p-1)\leq\gldim S_k(t, 2p)\leq 2(2p-\left\lceil \frac{p^1}{t} \right\rceil -\left\lceil \frac{p^1}{t} \right\rceil )=2(2p-2)=4(p-1).$$
So, $\gldim S_k(t, 2p)=4(p-1)$. 
\end{proof}

This result highlights another homological distinction between the case $p=2$ and larger primes. Specifically, for $p>2$, we have $\gldim S_k(2p, 2p)=\gldim S_k(p, 2p)$, whereas for $p=2$ $\gldim S_k(4, 4)=6\neq \gldim S_k(2, 4)=4$.

\subsubsection{Global dimension of Schur algebras $S(t, 2p)$ over arbitrary fields} \label{globaldimensionoverarbitraryfields}

Schur algebras are quasi-hereditary over any field (see for example \citep[Section 5]{cellularqhalgebras} and the references therein), hence they have finite global dimension. So, the global dimension of $S_k(t, 2p)$ coincides with the projective dimension of $DS_k(t, 2p)$. Since the latter has a base change property, this means that the global dimension of Schur algebras over a field is preserved under base change to an algebraically closed field. To make this precise, we use the following folklore lemma, widely known though rarely stated explicitly.

\begin{Lemma}
	Let $k$ be an arbitrary field and $A$ a finite-dimensional $k$-algebra. Then,  $$\injdim \overline{k}\otimes_k A=\injdim A,$$ where $\overline{k}$ denotes the algebraic closure of $k$. \label{lemmainvarianceofinjectivedimension}
\end{Lemma}
\begin{proof}
	See Proposition 2.1 of \cite{zbMATH03785104}.
\end{proof}

We can also make use of this lemma  to generalise Theorem \ref{thm3:3:2} and the previous results in this section to arbitrary fields.

\begin{Theorem}\label{maintheoremA}
	Let $k$ be an arbitrary field with positive characteristic $p$. Then, the following holds.
\begin{enumerate}
	\item $\gldim S_k(t, 2p)=4(p-1)$ for every $t=p, \ldots, 2p-1$.
	\item $(S_k(p, 2p), V^{\otimes 2p})$ is a relative $4(p-1)$-Auslander pair.
	\item $\Lambda_k(p, 2p)$ is Iwanaga-Gorenstein, and $\findim \Lambda_k(p, 2p)=2p-4$ when $p>2$.
\end{enumerate}
\end{Theorem}
\begin{proof} Let $\overline{k}$ be the algebraic closure of $k$.
	Since Schur algebras (over a field) have finite global dimension, we have $\gldim S_k(t, 2p)=\injdim S_k(t, 2p)$. By Lemma \ref{lemmainvarianceofinjectivedimension} and Corollary \ref{cor5dot3dot5},
	$\injdim S_k(t, 2p)=\injdim \overline{k}\otimes_k S_k(t, 2p)=\injdim S_{\overline{k}}(t, 2p)=4(p-1).$ 
	
	Part (2) follows from part (1) together with Theorem \ref{maintheoremarbitraryfield}.
	
	For part (3), recall from Theorem \ref{thm3:3:2} and Corollary \ref{finitisticlamdbazero} that $\L_{\overline{k}}(p, 2p)$ is Iwanaga-Gorenstein and $\findim \L_{\overline{k}}(p, 2p)=2p-4$ when $p>3$. Using the same arguments and the resolutions in Equations (\ref{equation18}), (\ref{equation19}) and (\ref{equation20}) we obtain that $\La_{\overline{k}}(3, 6)$ has injective dimension two (both as left and right module) when $p=3$. 
	
	 Observe that 
	 \begin{align}
	 	\overline{k}\otimes_k\La_k(p, 2p)&\cong \overline{k}\otimes_k \End_{S_k(p, 2p)}((k^p)^{\otimes 2p})^{op}\cong \End_{\overline{k}\otimes_k S_k(p, 2p)}(\overline{k}\otimes_k(k^p)^{\otimes 2p})^{op} \\&\cong \End_{S_{\overline{k}}(p, 2p)}(({\overline{k}}^p)^{\otimes 2p})^{op}=\La_{\overline{k}}(p, 2p).
	 \end{align} Assume that $p>2$. Since $\Lambda_{\overline{k}}(p, 2p)$ has finite injective dimension (as a left and as a right module), it follows by Lemma \ref{lemmainvarianceofinjectivedimension} that  $\Lambda_k(p, 2p)$ is Iwanaga-Gorenstein and $\findim \La_k(p, 2p)=\injdim \La_k(p, 2p)=\injdim \La_{\overline{k}}(p, 2p)=2p-4$.
\end{proof}

	\section*{Acknowledgements} 
	
This work began during the second author’s visit to Stuttgart in 2024 and continued during the first author’s visit to Oxford in 2025. The authors are grateful for the hospitality and support they received. The authors thank Stephen Donkin for bringing the work developed in \cite{zbMATH02005565} to their attention. The first-named author thanks Kevin Schlegel for useful discussions about Lemma 5.3.6. The authors thank the referees for their helpful comments and suggestions.


\bibliographystyle{alphaurl}
\bibliography{ref}

\end{document}